\documentclass[12 pt]{amsart}

\usepackage{amscd,amsmath,latexsym,amssymb,amscd,stmaryrd}
\usepackage[mathscr]{euscript}
\usepackage[all]{xy}
\usepackage{layout}
\usepackage{enumerate}
\usepackage{color}
\usepackage{hyperref}

\def\makeautorefname#1#2{\expandafter\def\csname#1autorefname\endcsname{#2}}
\makeautorefname{section}{Section}
\makeautorefname{theorem}{Theorem}
\makeautorefname{proposition}{Proposition}
\makeautorefname{lemma}{Lemma}
\makeautorefname{corollary}{Corollary}
\makeautorefname{definition}{Definition}
\makeautorefname{remark}{Remark}
\makeautorefname{example}{Example}

\newtheorem{theorem}{Theorem}[section]
\newtheorem{corollary}[theorem]{Corollary}
\newtheorem{lemma}[theorem]{Lemma}
\newtheorem{proposition}[theorem]{Proposition}
\theoremstyle{definition}
\newtheorem{definition}[theorem]{Definition}
\newtheorem*{thma}{Theorem A}
\newtheorem*{thmb}{Theorem B}

\newtheorem*{conjecture}{Conjecture}

\theoremstyle{definition}
\newtheorem{example}[theorem]{Example}
\newtheorem{remark}[theorem]{Remark}


\newcommand{\comments}[1]{}

\def\Z{\mathbb{Z}}

\def\C{\mathbb{C}}

\def\Q{\mathbb{Q}}
\def\R{\mathbb{R}}

\def\SS{\mathbb{S}}

\def\A{\mathcal{A}}

\def\H{\mathcal{H}}

\def\I{\mathcal{I}}

\def\lR{\mathcal{R}}
\def\U{\mathcal{U}}

\def\Fred{\text{Fred}}

\def\g{\mathfrak{g}}
\def\t{\mathfrak{t}}

\DeclareMathOperator{\Hom}{Hom}

\begin{document}

\title[Twisted equivariant K-theory of actions with maximal rank isotropy]
{Twisted equivariant K-theory of compact Lie group actions with
maximal rank isotropy}

\author[A.~Adem]{Alejandro Adem}
\address{Department of Mathematics,
University of British Columbia, Vancouver BC V6T 1Z2, Canada}
\email{adem@math.ubc.ca}

\author[J.~Cantarero]{Jos\'e Cantarero}
\address{Consejo Nacional de Ciencia y Tecnolog\'ia, 
Centro de Investigaci\'on en Matem\'aticas, A.C., Unidad M\'erida,
Parque Cient\'ifico y Tecnol\'ogico de Yucat\'an,
Carretera Sierra Papacal-Chuburn\'a Puerto Km 5.5,
Sierra Papacal, M\'erida, Yucat\'an 
CP 97302, Mexico}
\email{cantarero@cimat.mx}

\author[J.~M.~G\'omez]{Jos\'e Manuel G\'omez}
\address{Escuela de Matem\'aticas,
Universidad Nacional de Colombia, Medell\'in, AA 3840, Colombia}
\email{jmgomez0@unal.edu.co}
\thanks{The first author was partially supported by NSERC. The second author
was partially supported by CONACYT-SEP grant 242186. The third author acknowledges and thanks 
the financial support provided by the Max Planck Institute for Mathematics and by 
COLCIENCIAS through grant numbers FP44842-617-2014  
and FP44842-013-2018 of the 
Fondo Nacional de Financiamiento para la Ciencia, la Tecnolog\'ia y la Innovaci\'on.}

\begin{abstract}
We consider twisted equivariant K--theory for actions of a compact Lie group $G$
on a space $X$ where all the isotropy subgroups are connected and of maximal rank. We show that 
the associated rational spectral sequence \`a la Segal has a simple $E_2$--term 
expressible as invariants under the Weyl group of $G$. Namely, if $T$ is a maximal
torus of $G$, they are invariants of the $\pi_1(X^T)$-equivariant Bredon cohomology 
of the universal cover of $X^T$ with suitable coefficients. In the case of the inertia 
stack $\Lambda Y$ this term can be expressed using the 
cohomology of $Y^T$ and algebraic invariants associated to the Lie group and the twisting. 
A number of calculations are provided. In particular, we recover the rational Verlinde
algebra when $Y=\{*\}$.
\end{abstract}

\maketitle

\section{Introduction}
Let $G$ denote a compact Lie group with torsion--free fundamental group acting on a space $X$ 
such that all the isotropy subgroups are connected and contain a maximal torus for $G$, i.e., 
they have maximal rank. Let $T\subset G$ denote a maximal torus with Weyl group $W=N_G(T)/T$. 
In \cite{AG} it was shown that if the fixed--point set $X^T$ has the homotopy type of a finite
$W$--CW complex, then the rationalized complex equivariant K--theory of $X$
is a free module over the representation ring of $G$ of rank equal to 
$\sum_{i\ge 0} \textrm{dim}_{\Q} H^i(X^T;\Q)$. Moreover, assuming that every isotropy 
subgroup has torsion--free fundamental group, it was shown 
that if $\Lambda X$ denotes the \textsl{inertia stack} for the $G$--space $X$,
then $K_G^*(\Lambda X)\otimes\Q$ is a free module over $R(G)\otimes\Q$ of rank equal to
$2^r (\sum_{i\ge 0} \textrm{dim}_{\Q} H^i(X^T;\Q))$, where $r$ denotes the rank of the group $G$. 
The proofs of these results are based on 
the collapse over $\Q$ of a spectral sequence arising from a skeletal filtration of $X$ which was
first introduced by Segal \cite{Segal}. 

Based on this, a natural question arises: given a similar
hypothesis of maximal rank isotropy and a suitable
twisting $P$, under what conditions can we use analogous spectral sequence methods to
effectively compute $^PK^*_G(X)\otimes\Q$ or $^PK^*_G(\Lambda X)\otimes\Q$? 
In this paper we
provide a partial affirmative answer to this question. 

\begin{thma}
\textsl{Suppose $G$ is a compact Lie group acting on a space $X$ with 
connected maximal rank isotropy subgroups and with a fixed point $x_{0}$. 
Let $p\colon P\to X$ be a $G$-equivariant principal $PU(\H)$-bundle 
whose restriction to $x_{0}$ is trivial. 
Assume that $X^{T}$ is a $W$-CW complex of finite type.  Then there is a  
spectral sequence with $E_{2}$-term given by 
\begin{equation*}
E_{2}^{p,q}= \left\{
\begin{array}{ccc}
H^{p}_{\pi_{1}(X^{T})}(\widetilde{X^{T}};\lR_{\Q})^{W} 
&\text{ if } & q \text{ is even},\\
0& \text{ if } & q \ \ \text{ is odd}
\end{array}
\right.
\end{equation*}
converging to ${}^{P}K^{*}_{G}(X)\otimes \Q$.}
\end{thma}

Here $\lR_{\Q}$ denotes the coefficient system for Bredon cohomology defined
on the orbit category of $\pi_1(X^T)$ by $\pi_1(X^T)/V\mapsto (R(T)\otimes\Q)^V$ and 
$\widetilde{X^{T}}$ denotes the universal cover of $X^{T}$. This coefficient system is induced by 
an action of $\pi_1(X^T)$ on $R(T)$ that is constructed in Proposition \ref{pi1action} 
using the $G$-equivariant principal $PU(\H)$-bundle $p\colon P\to X$. 

The case of the inertia stack is especially interesting since $\Lambda (\{*\})\cong G$ with
the conjugation action. In particular, any $G$-space $X$ is equipped with a $G$-equivariant map 
$\Lambda X \to G$. The isomorphism classes of $G$-equivariant principal $PU(\H)$-bundles over a 
$G$-space $Y$ are in bijective correspondence with $H^3_G(Y;\Z)$. When $G$ is a compact, simple 
and simply connected Lie group, it is well known that $H^3_G(G;\Z) \cong \Z$. After fixing
a generator, pullback with respect to the $G$-equivariant map $\Lambda X \to G$ determines a 
$G$-equivariant $PU(\H)$-bundle $Q_n \to \Lambda X$ for each integer $n$. In this case we obtain 
the following.

\begin{thmb}
\textsl{Let $G$ be a compact, simple and simply connected Lie group of rank equal to $r$ and 
$n\ne 0$  an integer. Suppose that $X$ is a compact $G$-CW complex such that $G_{x}$ is a 
connected subgroup of maximal rank that has torsion free fundamental group for every 
$x\in X$ and that there is a point fixed by the action of $G$. 
Then the $E_{2}$-term in the spectral sequence computing 
${}^{Q_{n}}K_{G}^{*}(\Lambda X)\otimes \Q$ is given by 
\begin{equation*}
E_{2}^{p+r,q}= \left\{
\begin{array}{ccc}
\left[ H^{p}(X^{T};\Q)\otimes (R(T)^{sgn}\otimes \Q/J_{n}) \right]^{W}
&\text{ if } &p\ge 0\text{ and } q \text{ is even},\\
0& \text{ if } &q \ \ \text{ is odd.}
\end{array}
\right.
\end{equation*}} 
\end{thmb}

In the above theorem $R(T)^{sgn}\otimes \Q$ denotes $R(T) \otimes \Q$ with the action of 
$W$ given by
\[ w \bullet x = (-1)^{\ell(w)} w \cdot x, \]
where $w \cdot x$ denotes the usual action of $W$ on $R(T) \otimes \Q$ and $\ell(w)$ denotes
the length of $w$. Also, if $\{\alpha_{1},\dots, \alpha_{r}\}$ denotes a set of simple roots in the 
corresponding root system, then $J_{n}$ denotes the ideal in $R(T)^{sgn}$ generated by the elements 
$\theta_{\alpha_{i}}^{nd_{i}}-1$ for $1\le i\le r$, where $\theta_{\alpha_{i}}$ denotes 
the global root associated to $\alpha_{i}$ and $d_{i}$ is an integer. 
(See Section \ref{Section 4} for the details).

For the particular case of $X=\{*\}$ the previous theorem shows that the $E_{2}$-term in the 
spectral sequence computing ${}^{Q_{n}}K^{p}_{G}(G)\otimes \Q$ is such that 
$E_{2}^{r,q}=(R(T)^{sgn}\otimes \Q/J_{n})^{W}$ for $q$ even and $0$ in other cases. Therefore 
the spectral sequence collapses at the $E_{2}$-term for trivial reasons in this case. 
Moreover, if  $k\ge 0$ is an integer such that $k=n-h^{\vee}$, where $h^{\vee}$ 
is the dual Coxeter number of the group $G$, then we show that $(R(T)^{sgn}\otimes \Q/J_{n})^{W}$ 
can be identified with the rational Verlinde algebra 
$V_{k}(G)_{\Q}:= (R(G)\otimes \Q )/I_{k}$ at level $k=n-h^{\vee}$.  
In particular, when $p$ has the same parity as the rank of the group, 
we conclude that ${}^{Q_{n}}K^{p}_{G}(G)\otimes \Q$ is isomorphic 
as a module over $R(G)\otimes \Q$ to the Verlinde algebra $V_{k}(G)_{\Q}$ at 
level $k=n-h^{\vee}$, and we recover the 
celebrated result due to Freed--Hopkins--Teleman (see \cite[Theorem 1]{FHTIII}). 
We are also able to provide a complete calculation for the inertia space $\Lambda \SS^{\g}$, 
where $\SS^{\g}$ denotes the one point compactification of the Lie algebra $\g$ (see 
Subsection \ref{inertia sphere}). We observe that in these examples the spectral 
sequence collapses; we provide a number of other examples where this also holds, leading
us to formulate:

\begin{conjecture}
Under the hypotheses of Theorem B, the spectral sequence for computing 
$^{Q_n}K_G^*(\Lambda X)\otimes\Q$ always collapses at the $E_2$ level.
\end{conjecture} 

\medskip

This paper is organized as follows: in \S 2 we provide a detailed definition of twisted 
equivariant K--theory; in \S3 we define the spectral sequence \`a la Segal that will be used in 
our calculations; in \S 4 we discuss facts from root systems and cohomology that are required 
to provide a description of the $E_2$--term in the applications;
in \S5 we apply our approach to inertia spaces, providing a number of explicit examples; and 
finally in \S6 we provide an appendix on the Verlinde algebra where we explain how it can be 
identified in terms of the invariants from our calculations. We are grateful to the referee for providing
very helpful comments and suggestions.

\section{Definition of Twisted Equivariant K-theory}

In this section we briefly review the definition of twisted equivariant 
K-theory that we will use throughout this article. We remark that in this work all spaces in sight 
are $G$-CW complexes unless otherwise specified and $G$ denotes a compact Lie group.

We start by recalling the definition of $G$-equivariant principal 
$\Pi$-bundles. 

\begin{definition}
Suppose that $G$ is a compact Lie group and let $\Pi$  be a topological group. A 
$G$-equivariant principal $\Pi$-bundle consists of a $G$-equivariant map $p\colon E\to X$, 
where $G$ is acting on the left on both $E$ and $X$, together with 
a right action of $\Pi$ on $E$ that commutes with the left $G$-action in such a 
way that the map $p\colon E\to X$ is a principal $\Pi$-bundle. 
\end{definition}

Suppose that $X$ is a $G$-space. Let $U_{x}$ be a $G_{x}$-invariant open 
neighborhood of $x$. Recall that $U_{x}$ is called a slice through $x$ 
if the map 
\begin{align*}
\mu \colon G\times_{G_{x}}U_{x}\to X\\
[g,y]\mapsto gy
\end{align*} 
is a homeomorphism onto $GU_{x}$. In this case we can identify equivariantly 
the tube $GU_{x}$ with $G\times_{G_{x}}U_{x}$. As in \cite{Lashof}
the $G$-equivariant principal $\Pi$-bundles that we work 
with will satisfy the following local triviality condition.

\begin{definition}
A $G$-equivariant principal $\Pi$-bundle  $p\colon E\to X$ is $G$-locally trivial if for every 
$x\in X$ we can find a $G_{x}$-invariant open slice $U_{x}$ through $x$ and a 
local trivialization $\varphi_{x}\colon p^{-1}(GU_{x})\to G\times_{G_{x}} (U_{x}\times \Pi)$ 
of $p$ as a  principal $\Pi$-bundle which is also $G$-equivariant. 
Here the action of $G_{x}$ on $ U_{x}\times \Pi$  is given by:
\[
h\cdot (y, \sigma):=(hy,\gamma_{x}(h)\sigma)
\]
for a fixed continuous homomorphism
$\gamma_{x}\colon G_{x}\to \Pi$.
\end{definition}

In the above definition the homomorphism $\gamma_{x}$ is called a local representation 
and it is well defined up to conjugation. Such local representations 
can be constructed directly from the bundle $p\colon E\to X$ as follows. 
Assume that $x\in X$ and fix an element $e_{x}\in E_{x}$. For every $g\in G_{x}$ we can write 
$g\cdot e_{x}=e_{x}\cdot \gamma_{x}(g)$ for a unique $\gamma_{x}(g)\in \Pi$. 
The assignment $g\mapsto  \gamma_{x}(g)$  defines a local representation and 
is independent of the choice of $e_{x}$ up to conjugation. Such local representations 
play a key role in twisted equivariant K-theory as we shall see later on.
 
In this work we are mainly interested in the particular case where $\Pi =PU(\H)$ for a 
suitable Hilbert space $\H$, as such bundles can be used to twist equivariant K-theory. 
Fix a compact Lie group $G$ acting continuously on a space $X$. Let $\H$ be a separable 
infinite-dimensional Hilbert space. Denote by  $\Fred^{(0)}(\H)$ the space of 
self-adjoint degree one Fredholm operators $F\colon \H\oplus \H \to \H\oplus \H$ 
such that $F^{2}-I$ is a compact operator with the topology described in  
\cite[Definition 3.2]{AS}. The group $PU(\H)$ endowed with the compact-open topology 
acts continuously on $\Fred^{(0)}(\H)$ by conjugation and we have a central extension 
\[
1\to \SS^{1}\to U(\H)\to PU(\H)\to 1.
\]
If $K\subset G$ is a closed subgroup and $\gamma\colon K\to PU(\H)$ is a homomorphism, 
via pullback we obtain a compact Lie group $\widetilde{K}:=\gamma^{*}U(\H)$ that fits 
into a central extension making the following diagram commutative 
\[
\xymatrix{
1\ar[r]&\SS^1 \ar[r] \ar[d]^{\text{id}} & \widetilde{K} 
\ar[r] \ar[d]^{\widetilde{\gamma}} & K \ar[d]^{\gamma}\ar[r]& 1 \\
1\ar[r]& \SS^1 \ar[r]& U(\H) \ar[r]  & PU(\H)\ar[r]& 1.
}
\] 
Using the homomorphism $\widetilde{\gamma}\colon \widetilde{K}\to U(\H)$, we can see 
$\H$ as a $\widetilde{K}$-representation in such a way that the central circle $\SS^{1}$ 
acts by multiplication of scalars.  We say that $\gamma$ is a stable homomorphism 
if $\H$ contains infinitely many copies of all the irreducible representations of 
$\widetilde{K}$ on which $\SS^{1}$ acts by multiplication of scalars.
Following \cite[Definition 3.1]{BEJU} we have the next definition.

\begin{definition}
A $G$-stable principal $PU(\H)$-bundle is a $G$-equivariant principal $PU(\H)$-bundle 
$p\colon P\to X$ that is $G$-locally trivial and such that for every $x\in X$ the 
local representation $\gamma_{x}\colon G_{x}\to PU(\H)$ is a stable homomorphism.
\end{definition}

Given a  $G$-equivariant principal $PU(\H)$-bundle 
$p\colon P\to X$ we can associate to it a $G$-equivariant bundle of Fredholm operators 
$\Fred^{(0)}(P)$ over $X$ defined by 
\[
\pi\colon \Fred^{(0)}(P):=P\times_{PU(\H)}\Fred^{(0)}(\H)\to X.
\]
The bundle $\Fred^{(0)}(P)$ has a basepoint in each fiber. For each positive integer $n$, 
let $\Fred^{(-n)}(P)$ be the fiberwise $n$-th loop space of $\Fred^{(0)}(P)$.

\begin{definition}
Suppose that $p\colon P\to X$ is a $G$-stable $G$-equivariant principal 
$PU(\H)$-bundle and $n \geq 0$ is an integer. The $(-n)$-th $P$-twisted 
$G$-equivariant K-theory of $X$, denoted by ${}^{P}K_{G}^{-n}(X)$, 
is defined to be the group of $G$-equivariant homotopy classes of $G$-equivariant sections of 
$\pi\colon \Fred^{(-n)}(P)\to X$.
\end{definition}

Since there is a fiberwise $G$-homotopy equivalence $\Fred^{(-n)}(P) \to \Fred^{(-n-2)}(P)$, 
we can extend this definition to positive integers in a natural way.

\begin{example}
Suppose that $G$ is a compact Lie group and that $X=\{*\}$ is a point. 
Let $P\to \{*\}$ be a $G$-stable $G$-equivariant principal 
$PU(\H)$-bundle. Fix some element
$e\in P=PU(\H)$ and let $\gamma\colon G\to PU(\H)$ be the local representation 
obtained using the element $e$ as explained above. 
The homomorphism $\gamma$ determines the  $G$-equivariant principal 
$PU(\H)$-bundle $P$ up to isomorphism. In this case the associated bundle of Fredholm operators  
\[ \Fred^{(0)}(P):=P\times_{PU(\H)}\Fred^{(0)}(\H)\cong PU(\H)\times_{PU(\H)}\Fred^{(0)}(\H) \]
can be identified with $\Fred^{(0)}(\H)$. With this identification the action of 
$G$ on $\Fred^{(0)}(\H)$ is obtained via the homomorphism $\gamma\colon G\to PU(\H)$ 
and the conjugation action of $PU(\H)$ on $\Fred^{(0)}(\H)$. Also, a 
$G$-equivariant section of the bundle $\Fred^{(0)}(P)\to \{*\}$ corresponds precisely to an 
element in $\Fred^{(0)}(\H)^{G}$. By definition ${}^{P}K_{G}^{0}(\{*\})$ corresponds to the group 
$\pi_{0}(\Fred^{(0)}(\H)^{G})$.

Let $\widetilde{G}$ be the pullback of $U(\H)\to PU(\H)$ along $\gamma$ so that we 
have a central extension
\[
1\to \SS^{1}\to \widetilde{G}\stackrel{\tau} \to G\to 1.
\]
The group $\widetilde{G}$ comes equipped with a homomorphism 
$\widetilde{\gamma}\colon \widetilde{G}\to U(\H)$ that covers $\gamma$.
As explained before, $\H$ is a $\widetilde{G}$ representation in such a way that the central 
circle $\SS^{1}$ acts by multiplication of scalars. With this action the group $\widetilde{G}$ 
acts on $\Fred^{(0)}(\H)$ in such a way that the central circle $\SS^{1}$ acts trivially.  
Moreover, the space $\Fred^{(0)}(\H)^{\widetilde{G}}$ is equivalent 
to the space $\Fred^{(0)}(\H)^{G}$. Since the homomorphism $\gamma$ is stable 
the group $\pi_{0}(\Fred^{(0)}(\H)^{G})$ can be identified with 
the Grothendieck group of all complex representations of $\widetilde{G}$ 
on which the central $\SS^{1}$ acts by multiplication of scalars. This is by definition 
the $\tau$-twisted complex representation ring of $G$ that we denote by  $R^{\tau}(G)$. 
We conclude that  ${}^{P}K_{G}^{0}(\{*\})\cong R^{\tau}(G)$.

In this case we also have that $\Fred^{(1)}(P)$ is $G$-homotopy equivalent to 
$\Omega \Fred^{(0)}(\H)$, from here we see ${}^P K_G^1(\{ *\})$ as a direct summand of 
$K_{\tilde{G}}^1(\{ *\})$, which is known to vanish. Hence ${}^P K_G^1(\{ * \}) = 0$.

Suppose now that $G$ is connected. 
Fix $T\subset G$ a maximal torus and let $W=N_{G}(T)/T$ be the corresponding Weyl group. 
Consider the central extension 
\[
1\to \SS^{1}\to \widetilde{G}\stackrel{\tau} \to G\to 1
\]
associated to the bundle $P\to \{*\}$ as above. Let $\widetilde{T}=\tau^{-1}(T)$, a maximal 
torus in $\widetilde{G}$. The space $\widetilde{T}$ also fits into a central extension 
\[
1\to \SS^{1}\to \widetilde{T}\stackrel{\tau}{\rightarrow} T\to 1.
\]
The Weyl group of $\widetilde{G}$ can be identified with $W$ and thus 
$W$ acts naturally on $R^{\tau}(T)$. In a similar way as in the case of 
untwisted K-theory we have a natural isomorphism $R^{\tau}(G)\cong R^{\tau}(T)^{W}$. 
Therefore when $G$ is a compact, connected Lie group 
we have a natural isomorphism  
\[
{}^{P}K_{G}^{0}(\{*\})\cong R^{\tau}(T)^{W}.
\] 
\end{example}

In a similar way as was done in \cite[Proposition 6.3]{AS} we can associate 
to each $G$-equivariant principal $PU(\H)$-bundle $p\colon P\to X$  an equivariant cohomology 
class $\eta_{P}\in H^{3}_{G}(X;\Z)$. This assignment defines a one-to-one correspondence 
between isomorphism classes of $G$-equivariant principal $PU(\H)$-bundles over $X$ 
and cohomology classes in $H^{3}_{G}(X;\Z)$. 

\section{A twisted spectral sequence for actions with maximal rank isotropy}

In this section we study a spectral sequence for twisted equivariant K-theory
analogous to the classical Segal spectral sequence \cite{Segal}
for equivariant K--theory associated to an appropriate covering. It is a formal 
consequence of \cite[Theorem 22.4.4]{MS}. Variations on this spectral sequence appear in 
\cite{BEJU}, \cite{Douglas}, \cite{Dwyer} and \cite{FHT}. 

We will show that in the particular case when we have actions of compact Lie groups with the 
property that all the isotropy groups are connected, of maximal rank, the $E_2$--term can be
succinctly described using Bredon cohomology with suitable coefficients, in a manner analogous 
to what holds in the untwisted case (see \cite{AG}).

We start by setting up some notation that will be used throughout this section. 
Fix  a compact, connected 
Lie group $G$ and $T$ a maximal torus in $G$.
Let  $W=N_{G}(T)/T$ be the corresponding Weyl group. We will assume 
that $G$ acts continuously on a space $X$ with connected maximal rank subgroups. This means that 
for every $x\in X$ the isotropy subgroup $G_{x}$ is connected and 
contains a maximal torus in $G$.  With these hypotheses the action of 
$G$ on $X$ induces an action of $W=N_{G}(T)/T$ on $X^{T}$ by passing to the $T$-fixed points. 
Our main goal is to compute the $E_{2}$-term in a spectral sequence for twisted equivariant 
K-theory in terms of the $W$-action on $X^{T}$. We will assume that $X^{T}$ is a $W$-CW complex. 
By \cite[Theorem 2.2]{AG} this is  equivalent to assuming that $X$ is a $G$-CW complex.  
Furthermore we will assume that there is a point $x_{0}\in X$ fixed by the $G$-action 
and that $X^{T}$ is path--connected although this last condition can be 
removed with obvious modifications.

Let $p\colon P\to X$ be a $G$-equivariant principal $PU(\H)$-bundle that we assume 
$G$-stable from now on. We also assume that the restriction of $p \colon P\to X$ over the 
base point $x_{0}$ is trivial. If the group $G$ is such that $\pi_{1}(G)$ 
is torsion--free then this condition holds for any bundle as $H^{3}(BG;\Z)=0$ 
for such groups. Since $p$ is $G$-locally trivial, for every $x\in X$ we can find a
$G_{x}$-invariant open slice $U_{x}$ through $x$ and a 
trivialization 
$$\varphi_{x}\colon p^{-1}(GU_{x})\to G\times_{G_{x}} (U_{x}\times PU(\H))$$ 
of $p$ as a principal $PU(\H)$-bundle which is also $G$-equivariant 
via  the local representation $\gamma_{x}\colon G_{x}\to PU(\H)$. 
Since the bundle $P$ is assumed to be $G$-stable then the local representations $\gamma_{x}$ 
is injective for all $x\in X$.  Let $\widetilde{G}_{x}=\gamma_{x}^{*}U(\H)$. In this way we get 
a central extension 
\[
1\to \SS^{1}\to \widetilde{G}_{x}\stackrel{\tau_{x}} \to G_{x}\to 1
\]
for every $x\in X$.  Notice that these central extensions depend on the 
choices made above and this family of central extensions does not necessarily 
vary continuously as $x$ moves in $X^{T}$.

\begin{proposition}\label{principalbundle}
Suppose that $G$ is a compact connected Lie group that acts on a $G$-CW complex $X$ 
with connected maximal rank isotropy and that $p\colon P\to X$ is a $G$-equivariant principal 
$PU(\H)$-bundle that is $G$-stable. Associated to $p$ there is a locally trivial bundle 
$q\colon L\to X^{T}$ such that the fiber over $x$ can be identified with a maximal torus
of $\widetilde{G}_x$. Fur-thermore, there is an action of $W$ on $L$ such that $q$ is a 
$W$-equivariant map.
\end{proposition}
\begin{proof}[\bf Proof: ]
For each $x\in X^{T}$, fix a $G_{x}$-invariant open slice $U_{x}$ through $x$ in $X$,   
a local representation $\gamma_{x}\colon G_{x}\to PU(\H)$ and a $G$-equivariant trivialization 
\[
\varphi_{x} \colon p^{-1}(GU_{x})\to G\times_{G_{x}} (U_{x}\times PU(\H)).
\] 
We write elements in $G\times_{G_{x}} (U_{x}\times PU(\H))$ in the form 
$[g,(z,\sigma)]$, where $g\in G$, $z\in U_{x}$ and $\sigma \in PU(\H)$. Therefore 
in  $G\times_{G_{x}} (U_{x}\times PU(\H))$ we have 
$[g,(z,\sigma)]=[gh^{-1},(hz,\gamma_{x}(h)\sigma)]$ for all $h\in G_{x}$. 
If $x,y \in X^{T}$ are such that $U_{x}\cap U_{y}\ne \emptyset$,  
let $\rho_{x,y} \colon U_{x}\cap U_{y}\to PU(\H)$ be the transition function defined by the equation 
\[
\varphi_{x}(\varphi_{y}^{-1}[g,(z,\sigma)])=[g,(z,\rho_{x,y}(z)\sigma)]
\]
for all $z\in U_{x}\cap U_{y}$, all $g\in G$ and all $\sigma\in PU(\H)$. The transition functions 
and the local representations are compatible in the sense that 
\begin{equation}\label{eqcompatible}
\gamma_{y}(t)=\rho_{x,y}(tz)^{-1}\gamma_{x}(t)\rho_{x,y}(z)
\end{equation}
for all $z\in U_{x}\cap U_{y}$ and all $t\in T$. Given  $x\in X^{T}$ 
define $V_{x}:=U_{x}\cap X^{T}$. By \cite[Theorem 1.1]{Hauschild}  
the Weyl group of $G_{x}$ is $W_{x}$ so that 
each $V_{x}$ is $W_{x}$-invariant open set in $X^{T}$. Moreover, as the 
action of $G$ on $X$ has connected maximal rank isotropy subgroups, using 
\cite[Theorem 2.1]{Hauschild} we see that $V_{x}$ is a $W_{x}$-invariant slice 
through $x$ in $X^{T}$ and $(GU_{x})\cap X^{T}$ can be identified 
$W$-equivariantly with $W\times_{W_{x}}V_{x}$. 

On the other hand, notice that for each $x\in X^{T}$, the group 
$\widetilde{T}_{x}=(\gamma_{x})_{|}^{*}U(\H)$ is a maximal torus in $\widetilde{G}_{x}$.  
Recall that 
\[
\widetilde{T}_{x}=(\gamma_{x})_{|}^{*}U(\H)=\{(t,u)\in T\times U(\H) ~|~ \gamma_{x}(t)=\pi(u)\},
\]
where $\pi \colon U(\H)\to PU(\H)$ denotes the canonical map. Next we construct an action 
of $W_{x}$ on $\widetilde{T}_{x}$ by automorphisms. Let $n \in N_{G_{x}}(T)$ and $(t,u)\in \widetilde{T}_{x}$ so that 
$\gamma_{x}(t)=\pi(u)$. We define
\[
n\cdot (t,u):=(ntn^{-1}, \gamma_{x}(n)u\gamma_{x}(n)^{-1})\in \widetilde{T}_{x}.
\]
The above assignment is well defined because 
\[
\pi(\gamma_{x}(n)u\gamma_{x}(n)^{-1})=\gamma_{x}(n)\pi(u)\gamma_{x}(n)^{-1}=
\gamma_{x}(n)\gamma_{x}(t)\gamma_{x}(n)^{-1}=\gamma_{x}(ntn^{-1})
\]
so that $n\cdot (t,u)  \in \widetilde{T}_{x}$. It is easy to see that this defines a continuous 
action of $N_{G_{x}}(T)$ on $\widetilde{T}_{x}$ by group automorphisms. Moreover, for every 
$n \in N_{G_{x}}(T)$ we have a commutative diagram 
\[
\xymatrix{
1\ar[r]  &\SS^{1}\ar[d]_{\text{id}} \ar[r] & \widetilde{T}_x 
\ar[d]^{n\cdot } \ar[r]^{\tau_x} &T\ar[r] \ar[d]^{c_{n}} &1\\
1\ar[r] &\SS^{1} \ar[r] &\widetilde{T}_{x}\ar[r]^{\tau_{x}} & T\ar[r] &1.}
\]
In the above diagram $c_{n}$ denotes the map given by $c_{n}(t)=ntn^{-1}$ for all $t\in T$.
Since the group of automorphisms of $\widetilde{T}_x$ is discrete and $T$ is path-connected, 
this action factors through an action of $W_x=N_{G_x}(T)/T$ on $\widetilde{T}_x$. This is 
also an action by group automorphisms and which fits into an analogous commutative diagram.

Using this action we can consider the space $W\times_{W_{x}}(V_{x}\times \widetilde{T}_{x})$, 
where $W_{x}$ acts on $V_{x}\times \widetilde{T}_{x}$ by the assignment 
$w\cdot (z,t,u)=(wz, wt, \gamma_{x}(n)u\gamma_{x}(n)^{-1})$. Here $n\in N_{G_{x}}(T)$ 
is any element such that $w=[n]$. We denote elements in 
$W\times_{W_{x}}(V_{x}\times \widetilde{T}_{x})$ 
in the form $[w, (z,t,u)]$ with $w\in W$, $z\in V_{x}$ and $(t,u)\in \widetilde{T}_{x}$. 
Suppose now that $z\in V_{x}\cap V_{y}$ and assume that $(z,t,u)\in V_{x}\times \widetilde{T}_{x}$. 
We claim that $(z,t,\rho_{x,y}(z)^{-1}u\rho_{x,y}(z))\in V_{y}\times \widetilde{T}_{y}$. 
Indeed, using (\ref{eqcompatible}) we have 
\[
\pi(\rho_{x,y}(z)^{-1}u\rho_{x,y}(z))=\rho_{x,y}(z)^{-1}\pi(u)\rho_{x,y}(z)
=\rho_{x,y}(z)^{-1}\gamma_{x}(t)\rho_{x,y}(z)=\gamma_{y}(t).
\]
With this in mind we define 
\[
L = \left( \coprod_{x\in X^{T}} W\times_{W_{x}}(V_{x}\times \widetilde{T}_{x}) \right) \Big{/} {\sim}.
\]
Here if $[w, (z,t,u)]\in W\times_{W_{x}}(V_{x}\times \widetilde{T}_{x})$ 
we define 
\[
[w, (z,t,u)]\sim [w,(z,t,\rho_{x,y}(z)^{-1}u\rho_{x,y}(z))]\in 
W\times_{W_{y}}(V_{y}\times \widetilde{T}_{y}).
\] 
This is well defined by the above comment. We denote by $\llbracket w, (z,t,u)\rrbracket$ 
the equivalence class in $L$ of the element $[w, (z,t,u)]$. 
With this definition the map $q\colon L\to X^{T}$  given 
by $q(\llbracket w, (z,t,u)\rrbracket)=wz$ is a locally trivial bundle.

To finish, we endow $L$ with an action of $W$. Suppose that 
$[w, (z,t,u)]\in W\times_{W_{x}}(V_{x}\times \widetilde{T}_{x})$ and that 
$w'\in W$. Then $[w'w, (z,t,u)]\in W\times_{W_{x}}(V_{x}\times \widetilde{T}_{x})$ 
and we define 
\[
w'\cdot \llbracket w, (z,t,u)\rrbracket=\llbracket w'w, (z,t,u)\rrbracket.
\] 
It is straightforward to check that this defines a continuous action of $W$ 
on $L$ and that $q \colon L\to X^{T}$ is a $W$-equivariant map. 
\end{proof}

\medskip

For the next proposition, note that any one-dimensional complex representation of $T$ is
a unit in the ring $R(T)$, hence multiplication by such a representation determines an 
automorphism of $R(T)$. Under tensoring of representations, the set $\Hom(T,\SS^1)$ of 
one-dimensional complex representations of $T$ is a subgroup of the group of units $R(T)^{\times}$.

\begin{proposition}\label{pi1action}
Suppose that $p\colon P\to X$ is $G$-equivariant principal $PU(\H)$-bundle such 
that the restriction of $p\colon P\to X$ to $x_{0}$ is trivial. 
Then $p\colon P\to X$ induces an action of $\pi_{1}(X^{T})$ on $R(T)$ given by 
a $W$-equivariant homomorphism   
\[
\phi_{P}\colon \pi_{1}(X^{T})\to \Hom(T,\SS^{1}).
\]
\end{proposition}
\begin{proof}[\bf Proof: ]
Let $q\colon L\to X^{T}$ be the bundle constructed in 
Proposition \ref{principalbundle}. Thus for every $x\in X^{T}$ we have  a local representation 
$\gamma_{x}\colon G_{x}\to PU(\H)$ together with central extensions 
\[
1\to \SS^{1}\to \widetilde{G}_{x}\stackrel{\tau_{x}}{\rightarrow} G_{x}\to 1
\]
in such a way that $\widetilde{T}_{x}=q^{-1}(x)$ is a maximal torus in $\widetilde{G}_{x}$.  
The space  $\widetilde{T}_{x}$ fits into a central extension of the form 
\[
1\to \SS^{1}\to \widetilde{T}_{x}\stackrel{\tau_{x}}{\rightarrow} T\to 1.
\]
For every $x\in X^{T}$ the above central extension is trivializable as $H^{3}(BT;\Z)=0$. 
All possible identifications of $\widetilde{T}_{x}$ with  $T\times \SS^{1}$
can be used to construct an action of $\pi_{1}(X^{T})$ on $R(T)$ via holonomy. 
To see this notice that associated to the bundle $L$ we have a 
covering space $r \colon L^{aut}\to X^{T}$  whose fiber over $x$ is the set of automorphisms 
of central extensions from $\widetilde{T}_x$ to $\widetilde{T}_x$. 
Let $\alpha\colon [0,1]\to X^{T}$ be any loop in $X^{T}$ based at $x_{0}$. 
Via holonomy associated to $\alpha$ there is an isomorphism of central 
extensions
\[
\xymatrix{
1\ar[r]  &\SS^{1}\ar[d]_{\text{id}} \ar[r] & \widetilde{T}_{x_{0}} 
\ar[d]^{\Phi_{\alpha}}\ar[r]^{\tau_{x_{0}}}&T\ar[r] \ar[d]^{\text{id}}&1\\
1\ar[r] &\SS^{1} \ar[r] &\widetilde{T}_{x_{0}} \ar[r]^{\tau_{x_{0}}}&T\ar[r] &1.}
\]
This isomorphism only depends on the class $[\alpha]\in \pi_{1}(X^T)$. 

The action of $W$ on $L$ constructed in Proposition \ref{principalbundle} induces an action 
of $W$ on $L^{aut}$. Namely, if $w \in W$ and $\varphi$ is an automorphism
of the central extension $\widetilde{T}_x$, we consider $ w \cdot \varphi = w \varphi w^{-1}$.
The map $r$ is $W$-equivariant with respect to this action, hence 
$\Phi_{w\alpha}=w\Phi_{\alpha}w^{-1}$ for any $w\in W$. In other words, 
the assignment $[\alpha]\mapsto \Phi_{\alpha}$ is $W$-equivariant with 
$W$ acting by conjugation on the group of automorphisms of the central extension 
$\widetilde{T}_{x_{0}}$. Observe that the group of automorphisms of the central 
extension $\widetilde{T}_{x_{0}}$ can be identified $W$-equivariantly 
with the group $\Hom(T,\SS^{1})$ as the restriction of $p$ over $x_{0}$ is 
assumed to be trivial. Thus the above assignment defines a $W$-equivariant 
group homomorphism 
\[
\phi_{P}\colon \pi_{1}(X^{T})\to \Hom(T,\SS^{1}).  \qedhere
\]
\end{proof}

\medskip

The homomorphism $\phi_{P}\colon \pi_{1}(X^{T})\to \Hom(T,\SS^{1})$ associated to a
$G$-equivariant principal $PU(\H)$-bundle $p\colon P\to X$ constructed in the 
previous proposition is a key part in the calculation of the
spectral sequence for twisted equivariant K-theory as we shall see below. 
On the other hand, notice that $\Hom(T,\SS^{1})\cong H^{2}(BT;\Z)$ as an abelian group. 
Therefore the homomorphism $\phi_{P}$ associated to the bundle $p\colon P\to X$  
induces a group homomorphism 
$\bar{\phi}_{P}\colon H_{1}(X^{T};\Z)\to \Hom(T,\SS^{1})\cong H^{2}(BT;\Z)$.  Using the 
universal coefficient theorem we can identify $\bar{\phi}_P$ with an element in 
$H^{1}(X^{T};H^{2}(BT;\Z))\cong H^{1}(X^{T};\Z)\otimes H^{2}(BT;\Z)$ which 
by abuse of notation we denote also by 
$\phi_{P}$. On the other hand, let 
$$h_{T}\colon H^{3}_{G}(X;\Z)\to H^{1}(X^{T};\Z)\otimes H^{2}(BT;\Z)$$ 
be the composite of the restriction map to $T$-fixed points 
\[
r_{T}\colon H^{3}_{G}(X;\Z)\to H_{T}^{3}(X^{T};\Z)\cong H^{3}(X^{T}\times BT;\Z)
\]
with the projection map 
\[
\pi_{T}\colon H^{3}(X^{T}\times BT;\Z)\cong H^{3}(X^{T};\Z)\oplus 
H^{1}(X^{T};\Z)\otimes H^{2}(BT;\Z) 
\to H^{1}(X^{T};\Z)\otimes H^{2}(BT;\Z).
\]
Using arguments similar to those used in \cite[Section 2.2]{Meinrenken} we can 
deduce the next proposition.

\begin{proposition}
Let $\eta_{P}\in H_{G}^{3}(X;\Z)$ be the cohomology class corresponding to the 
bundle $p\colon P\to X$. Then the element 
$h_{T}(\eta_{P})\in H^{1}(X^{T};\Z)\otimes H^{2}(BT;\Z)$ corresponds to $\phi_{P}$.
\end{proposition}

Suppose that $G$ acts on a space $X$ with connected maximal rank isotropy and 
that there is a  point $x_{0}\in X$ fixed by $G$ that we take as the base point. 
Let $\pi\colon \widetilde{X^{T}}\to X^{T}$ be 
the universal cover of $X^{T}$. Then the group $\pi_{1}(X^{T}):=\pi_{1}(X^{T},x_{0})$ 
acts on $\widetilde{X^{T}}$  by deck transformations.  
Fix a point $\tilde{x}_{0}$ in $\pi^{-1}(x_{0})$.   
For every $w\in W$, there is a unique map 
$\widetilde{w}\colon \widetilde{X^{T}} \to \widetilde{X^{T}}$ 
that is a lifting of the action map $w\colon X^{T}\to X^{T}$ and that satisfies 
$\widetilde{w}(\tilde{x}_{0})=\tilde{x}_{0}.$ 
Such a map exists and is unique by the lifting theorem for covering spaces. 
This assignment defines an action of $W$ on 
$\widetilde{X^{T}}$ in such a way that the map $\pi\colon \widetilde{X^{T}}\to X^{T}$  
is $W$-equivariant. On the other hand, observe that since the $G$-action leaves $x_{0}$ 
fixed, then  $W$ also fixes $x_{0}$ and therefore  
$W$ acts on $\pi_{1}(X^{T})$ by automorphisms under the assignment 
$w\cdot[\alpha]=[w\cdot \alpha]$. If we identify $\pi_{1}(X^{T})$  with the group 
of deck transformations on $\widetilde{X^{T}}$, this action corresponds to the 
conjugation action of $W$ on the group of deck transformations. 

With this action in mind, we can construct the semi-direct product  
$\widetilde{W}:=\pi_{1}(X^{T})\rtimes W$. We write elements in $\widetilde{W}$ in the 
form $(a,w)$ with $a\in \pi_{1}(X^{T})$ and $w\in W$. We have an induced 
action of $\widetilde{W}$ on $\widetilde{X^{T}}$. Explicitly, if 
$\tilde{x}\in \widetilde{X^{T}}$ and $(a,w)\in \widetilde{W}$ then
\[ (a,w)\cdot \tilde{x}:= a\cdot (w\cdot \tilde{x}) . \]
Moreover, if we let $\pi_{1}(X^{T})$ act trivially on $X^{T}$ then the action of 
$\widetilde{W}$ on $X^{T}$ is such that the projection map $\pi\colon \widetilde{X^{T}}\to X^{T}$ 
is $\widetilde{W}$-equivariant. 

Assume that $x\in X^{T}$ is a point and let $\tilde{x}\in \widetilde{X^{T}}$ be any point 
such that $\pi(\tilde{x})=x$. Suppose that $(a,w)\in \widetilde{W}_{\tilde{x}}$; that is, 
$(a,w)\cdot \tilde{x}=a\cdot (w\cdot \tilde{x})=\tilde{x}$. Applying $\pi$ to the previous equation 
we obtain $wx=x$ and thus $w\in W_{x}$. In addition, $a$ is a deck transformation that satisfies 
$a\cdot (w\cdot \tilde{x})=\tilde{x}$. Since deck transformations are uniquely 
determined by their values at a point we conclude that for every $w\in W_{x}$ there is a unique 
element $a\in \pi_{1}(X^{T})$ such that $(a,w)\in \widetilde{W}_{\tilde{x}}$. If $w\in W_{x}$ we 
denote by $a_{w}\in \pi_{1}(X^{T})$ the unique element such that 
$(a_{w},w)\in \widetilde{W}_{\tilde{x}}$. It can be seen that the assignment $w\mapsto a_{w}$ 
defines a cocycle on $W_{x}$ with values on $\pi_{1}(X^{T})$.

\medskip

Next we want to study a spectral sequence \`a la Segal computing ${}^{P}K_{G}^{*}(X)$ 
for a $G$-stable $G$-equivariant principal $PU(\H)$-bundle  $p\colon P \to X$.  
In the context of twisted equivariant K-theory, the spectral sequence is easier 
to describe using $G$-invariant open covers of $X$. 

\begin{definition}
Let $X$ be a $G$-space. We say that a $G$-invariant open subspace $U$ of $X$
is a contractible slice if there exists $x \in U$ such that the inclusion map
$Gx \hookrightarrow U$ is a $G$-homotopy equivalence.
\end{definition}

\begin{definition}
A $G$-equivariant good cover of a $G$-space $X$ is a cover $\U=\{U_{i}\}_{i\in \I}$ of $X$
by $G$-invariant open subsets such that the following two conditions hold:
\begin{itemize}
\item $\I$ is a well ordered set.
\item For every sequence $i_{1}\le \cdots\le i_{p}$ of elements in $\I$, if 
$U_{i_{1},\dots, i_{p}}:=U_{i_{1}}\cap \cdots \cap U_{i_{p}}$ is nonempty, then
it is a contractible slice.
\end{itemize}
\end{definition}

These covers are also called contractible slice covers. Note that in particular,
whenever $U_{i_{1},\dots, i_{p}}$ is nonempty, there exists some element 
$x_{i_{1},\dots,i_{p}}\in U_{i_{1},\dots, i_{p}}$ such that the inclusion map  
$Gx_{i_{1},\dots,i_{p}}\hookrightarrow U_{i_{1},\dots,i_{p}}$ is a 
$G$-homotopy equivalence and therefore $U_{i_{1},\dots,i_{p}}
\simeq G/G_{x_{i_{1},\dots,i_{p}}}$. 

When $G$ is a compact Lie group and $X$ is a $G$-ANR, the existence of $G$-equivariant
good covers is guaranteed by the work in \cite{Antonyan}. In particular, $G$-equivariant good 
covers exist for finite dimensional $G$-CW complexes when $G$ is a compact Lie group.

\begin{lemma}\label{one-to-one cover}
Suppose that $G$ acts on $X$ with connected maximal rank isotropy subgroups. 
Then there is a one-to-one correspondence between $G$-equivariant good covers on 
$X$ and $\widetilde{W}$-equivariant good covers on $\widetilde{X^{T}}$.
\end{lemma}
\begin{proof}[\bf Proof: ]
Suppose that $\U=\{U_{i}\}_{i\in \I}$ is a $G$-equivariant. For each $i \in \I$, 
the set $U_{i}^{T}:=U_{i}\cap X^{T}$ is open in $X^{T}$ and $W$-invariant. By assumption, for each 
sequence $i_{1}\le \cdots\le i_{p}$ of elements in $\I$ with $U_{i_{1},\dots,i_{p}}$ 
nonempty we can find an element $x_{i_{1},\dots, i_{p}}\in U_{i_{1},\dots,i_{p}}$ such that 
the inclusion map $Gx_{i_{1},\dots, i_{p}}\hookrightarrow U_{i_{1},\dots,i_{p}}$ is a 
$G$-homotopy equivalence. Since $G_{x_{i_{1},\dots, i_{p}}}$ is a subgroup of maximal rank
and all maximal tori in $G$ are conjugate, then after replacing $x_{i_{1},\dots, i_{p}}$  with
 $gx_{i_{1},\dots, i_{p}}$ for a suitable $g$,
we may assume without loss of generality that $x_{i_{1},\dots, i_{p}}\in X^{T}$. Then 
$x_{i_{1},\dots, i_{p}}\in U_{i_{1},\dots,i_{p}}\cap X^{T}=U_{i_{1},\dots,i_{p}}^{T}$. Moreover, 
the $G$-homotopy equivalence $Gx_{i_{1},\dots, i_{p}}\hookrightarrow U_{i_{1},\dots,i_{p}}$  
induces a $W$-homotopy equivalence  
$Wx_{i_{1},\dots, i_{p}}\hookrightarrow U_{i_{1},\dots,i_{p}}^{T}$ after passing to $T$-fixed 
points. Thus $\U^{T}=\{U_{i}^{T}\}_{i\in \I}$ is a $W$-equivariant good cover of 
$X^{T}$. 

Conversely, if $\U^{T}=\{U_{i}^{T}\}_{i\in \I}$  is a $W$-equivariant 
good cover of $X^{T}$ then for every $i\in I$ we can define $U_{i}=\cup_{g\in G}gU^{T}_{i}$.  
Then $\U=\{U_{i}\}_{i\in \I}$ is an open cover of $X$ by $G$-invariant sets.  
For each sequence $i_{1}\le \cdots\le i_{p}$ of elements in $\I$ with $U_{i_{1},\dots,i_{p}}^{T}$ 
nonempty, the $W$-homotopy equivalence  
$Wx_{i_{1},\dots, i_{p}}\hookrightarrow U_{i_{1},\dots,i_{p}}^{T}$ induces a 
$G$-homotopy equivalence $Gx_{i_{1},\dots, i_{p}}\hookrightarrow U_{i_{1},\dots,i_{p}}$ 
by  \cite[Theorem 2.1]{Hauschild}. This shows that $G$-equivariant good covers on $X$ are 
in one-to-one correspondence with $W$-equivariant good covers on $X^{T}$.

Suppose now that $\U^{T}=\{U_{i}^{T}\}_{i\in \I}$  is a $W$-equivariant 
good cover on $X^{T}$. For each $i\in \I$ define  
$\widetilde{U_{i}}=\pi^{-1}(U^{T}_{i})$. Then $\widetilde{\U}=\{\widetilde{U}_{i}\}_{i\in\I}$  
is an open cover of $\widetilde{X^{T}}$  by $\widetilde{W}$-invariant sets.  
Fix a sequence  $i_{1}\le \cdots\le i_{p}$ of elements in $\I$
with $U_{i_{1},\dots,i_{p}}^{T}$  nonempty  so that there is a 
$W$-homotopy equivalence  
$Wx_{i_{1},\dots, i_{p}}\hookrightarrow U_{i_{1},\dots,i_{p}}^{T}$ for some $x_{i_{1},\dots, i_{p}}$.
Fix $\tilde{x}_{i_{1},\dots, i_{p}}\in \widetilde{U}_{i_{1},\dots,i_{p}}$ with 
$\pi(\tilde{x}_{i_{1},\dots, i_{p}})=x_{i_{1},\dots, i_{p}}$. Notice that 
the restriction map  $\pi_{|}\colon \widetilde{U}_{i_{1},\dots,i_{p}}\to U_{i_{1},\dots,i_{p}}^{T}$ 
is a covering space.
Using the lifting property we can lift the $W$-homotopy equivalence 
$Wx_{i_{1},\dots, i_{p}}\hookrightarrow U_{i_{1},\dots,i_{p}}^{T}$ to a 
$\widetilde{W}$-equivariant homotopy equivalence 
$\widetilde{W}\tilde{x}_{i_{1},\dots, i_{p}}
\hookrightarrow \widetilde{U}_{i_{1},\dots,i_{p}}$ proving that $\widetilde{\U}$ is a 
$\widetilde{W}$-equivariant good cover on $\widetilde{X^{T}}$.  

Conversely, if $\widetilde{\U}=\{\widetilde{U}_{i}\}_{i\in\I}$  is a $\widetilde{W}$-equivariant 
good cover on $\widetilde{X^{T}}$, then it is easy to see that we can get a $W$-equivariant good 
cover on $X^{T}$ by defining $U^{T}_{i}=\widetilde{U}_{i}/\pi_{1}(X^{T})$. Thus $W$-equivariant 
good covers on $X^{T}$ are in one-to-one correspondence with 
$\widetilde{W}$-equivariant good covers on $\widetilde{X^{T}}$ and the lemma follows. 
\end{proof}

\medskip

Suppose now that $X$ is a $G$-space for which we can find 
a $G$-equivariant good open cover $\U=\{U_{i}\}_{i\in \I}$
and $p \colon P \to X$ is a $G$-equivariant principal $PU(\H)$-principal bundle. 
If  $i_{1}\le \dots\le i_{p}$ is a sequence of elements 
in  $\I$ we denote by $P_{i_{1},\dots, i_{p}}$ 
the restriction of $P$ to $U_{i_{1},\dots,i_{p}}$.  
Associated to the cover $\U$ we have a spectral sequence \`a la Segal that is constructed in a 
similar way as in \cite[Section 4]{BEJU}. The $E_{1}$-term in this spectral sequence is given by 
\begin{equation}\label{Segalss}
E_{1}^{p,q}=\prod_{i_{1}\le \dots\le i_{p}}
{}^{P_{i_{1},\dots, i_{p}}}K_{G}^{q}(U_{i_{1},\dots, i_{p}}).
\end{equation}
The differential $d_{1}\colon E_{1}^{p,q}\to E_{1}^{p+1,q}$ is given by the alternating sum 
of the different restriction maps 
\[
{}^{P_{i_{1},\dots,\hat{i}_{j},\dots, i_{p+1}}}
K_{G}^{q}(U_{i_{1},\dots,\hat{i}_{j},\dots, i_{p+1}}) 
\to {}^{P_{i_{1},\dots,i_{p+1}}}K_{G}^{q}(U_{i_{1},\dots, i_{p+1}})
\]
for all $1\le j\le p+1$. Here we use the usual convention that $\hat{i}_{j}$ means that 
the index $i_{j}$ is removed.  Our next goal is to identify the $E_{2}$-term of this 
spectral sequence with a suitable Bredon cohomology group of $\widetilde{X^{T}}$.  

Recall that if the restriction of $p$ to $x_{0}$ is trivial, then  
by Proposition \ref{pi1action} associated to the bundle $p\colon P\to X$
we have an action of $\pi_{1}(X^{T})$ on $\Hom(T,\SS^{1})$ that is 
$W$-equivariant. Therefore this action can be extended to an action of 
$\widetilde{W}=\pi_{1}(X^{T})\rtimes W$ on $R(T)$. We now consider the
coefficient system 
\[
\lR:=H^{0}(-;R(T)).
\]
defined on the orbit category of $\widetilde{W}$. Explicitly the value of this coefficient 
system at an orbit of the form $\widetilde{W}/\widetilde{W_{i}}$ are
\[
\lR(\widetilde{W}/\widetilde{W_{i}})=R(T)^{\widetilde{W_{i}}}.
\]
Note that $\lR$ can also be seen as a functor from the orbit category of $\widetilde{W}$
to the category of $R(G)$-modules.

\begin{theorem}\label{identifycontract}
Suppose that $p \colon P\to X$ is a $G$-equivariant principal $PU(\H)$-bundle that 
is $G$-stable and such that its restriction to base point $x_{0}$ is trivial. 
Assume that $U$ is a contractible slice and let $\widetilde{U}=\pi^{-1}(U\cap X^{T})$. 
Then for every even integer $q$ there is an isomorphism of 
$R(G)$-modules 
\[
\psi_{U}\colon {}^{P_{U}}K_{G}^{q}(U)\to H_{\widetilde{W}}^{0}(\widetilde{U};\lR).
\]
Moreover, if $U$ and $V$ are two $G$-equivariant slices with $U\subset V$  
the following diagram commutes 
\begin{equation*}
\xymatrix{
{}^{P_{V}}K_{G}^{q}(V)\ar[d]\ar[r]^{\psi_{V}} & 
H^{0}_{\widetilde{W}}(\widetilde{V};\lR) \ar[d]\\
{}^{P_{U}}K_{G}^{q}(U)\ar[r]^{\psi_{U}} & 
H^{0}_{\widetilde{W}}(\widetilde{U};\lR)
}
\end{equation*}
where the vertical maps are the restriction maps. 
\end{theorem}
\begin{proof}[\bf Proof: ]
Suppose that $U$ is a contractible slice. Then we can find some $x\in U$ 
such that the inclusion map $Gx\hookrightarrow U$ is a  $G$-homotopy 
equivalence.  After replacing $x$ with $gx$ for a suitable $g$, we may assume that 
$x\in X^{T}$ so that $x\in U\cap X^{T}$.  Let $P_{x}$ be the restriction of 
$P$ to the orbit $Gx$.  Notice that the inclusion map $Gx\hookrightarrow U$  induces an 
isomorphism
\[
{}^{P_{U}}K_{G}^{0}(U)\cong {}^{P_{x}}K_{G}^{0}(Gx)\cong R^{\tau_{x}}(G_{x}).
\]
Here $\tau_{x}$ denotes the  central extension of the group 
$G_{x}$ defined by the local representation $\gamma_{x}\colon G_{x}\to PU(\H)$ 
used in the construction of $L$ and $R^{\tau_{x}}(G_{x})$ denotes the $\tau_{x}$-twisted 
representation ring of $G_{x}$. As in the case of untwisted K-theory, 
the inclusion map $T\subset G_{x}$ induces an isomorphism
\[
R^{\tau_{x}}(G_{x})\cong R^{\tau_{x}}(T)^{W_{x}}. 
\]
In the above equation by abuse of notation we also denote by 
$\tau_{x}$ the central extension 
\[
1\to \SS^{1}\to \widetilde{T}_{x}\stackrel{\tau_{x}}{\rightarrow} T\to 1.
\]
Note that $\widetilde{U}=\pi^{-1}(U\cap X^{T})$ is a 
$\widetilde{W}$-invariant open set in $\widetilde{X^{T}}$.  
Let $\tilde{x}$ be any point in $\widetilde{X^{T}}$ 
such that $\pi(\tilde{x})=x$. Thus $\tilde{x}\in \widetilde{U}$ and  the inclusion 
map $\widetilde{W}\tilde{x}\hookrightarrow \widetilde{U}$ is a $\widetilde{W}$-equivariant 
homotopy equivalence. Let  $ \tilde{\beta}\colon [0,1]\to \widetilde{X^{T}}$ be any path in 
$\widetilde{X^{T}}$ from $\tilde{x}_{0}$ to $\tilde{x}$, where $\tilde{x}_{0}$ is the base 
point in $\widetilde{X^{T}}$. Such a path exists and is unique up to path homotopy 
since $\widetilde{X^{T}}$ is simply connected. Let $\beta=\pi\circ \tilde{\beta}$. 
Notice that $\beta$ is a path in 
$X^{T}$ from $x_{0}$  to $x$ and   by definition $\tilde{\beta}$ is the unique lifting of 
$\beta$ to  $\widetilde{X^{T}}$ such that $\tilde{\beta}(0)=\tilde{x}_{0}$. The 
path $\beta$ induces an isomorphism of central extensions 
$\Phi_{\beta}\colon \widetilde{T}_{x_{0}}\to \widetilde{T}_{x}.$
Moreover, this isomorphism only depends on the path-homotopy class of $\beta$.
Therefore $\beta$ induces an isomorphism 
\[
\Phi_{\beta}\colon R^{\tau_{x}}(T)\rightarrow R^{\tau_{x_{0}}}(T)
\]
On the other hand, fix a trivialization
$\widetilde{T}_{x_{0}}\cong T\times \SS^{1}$.  
Thus we have an identification 
$R^{\tau_{x_{0}}}(T)=R(T)$. We show next that $\Phi_{\beta}$  induces an isomorphism 
\begin{equation}\label{isom1}
\Phi_{\beta}\colon R^{\tau_{x}}(T)^{W_{x}}\rightarrow
R(T)^{\widetilde{W}_{\tilde{x}}}.
\end{equation}
For this we give $R^{\tau_{x}}(T)$ an action of $\widetilde{W}_{\tilde{x}}$ 
as follows. Elements in $\widetilde{W}_{\tilde{x}}$ are pairs of the form 
$(a_{w},w)$ with $w\in W_{x}$ and $a_{w}\in \pi_{1}(X^{T})$. 
We let 
$(a_{w},w)$ act on $p\in R^{\tau_{x}}(T)$ by the assignment 
\[
(a_{w},w)\cdot p=w\cdot p;
\]
that is, the part corresponding to $\pi_{1}(X^{T})$ acts trivially on $R^{\tau_{x}}(T)$.
With this action we clearly have 
$R^{\tau_{x}}(T)^{W_{x}}
=R^{\tau_{x}}(T)^{\widetilde{W}_{\tilde{x}}}$. 
To prove (\ref{isom1}) it suffices to prove that the map  $\Phi_{\beta}$ is 
$\widetilde{W}_{\tilde{x}}$-equivariant.  To see this 
fix an element $(a_{w},w)\in \widetilde{W}_{\tilde{x}}$ so that 
$w\in W_{x}$ and $a_{w}\in \pi_{1}(X^{T})$. We need to prove that
\[ \Phi_{\beta}(w\cdot p)=\Phi_{\beta}((a_{w,}w)\cdot p)=(a_{w},w)\cdot \Phi_{\beta}(p) \]  
for every $p\in R^{\tau_{x}}(T)$. Let $\alpha\colon [0,1]\to X^{T}$ 
be a loop based at $x_{0}$ such that $a_{w}=[\alpha]\in \pi_{1}(X^{T})$. Let 
$\tilde{\alpha}\colon [0,1]\to \widetilde{X^{T}}$ be the unique lifting of $\alpha$ to 
$\widetilde{X^{T}}$ such that $\tilde{\alpha}(0)=\tilde{x}_{0}$.  Let 
$D_{a_{w}}\colon \widetilde{X^{T}}\to \widetilde{X^{T}}$ be the deck transformation 
that corresponds to $a_{w}$.  Since $(a_{w},w)\in \widetilde{W}_{\tilde{x}}$ 
we have 
\[
(a_{w},w)\cdot \tilde{x}
=a_{w}\cdot(w\cdot \tilde{x})=D_{a_{w}}(w\cdot \tilde{x})=\tilde{x}; 
\]
that is, $w\cdot \tilde{x}=D_{a_{w}^{-1}}(\tilde{x})$. Consider the paths 
$w\cdot \tilde{\beta}$ and $\tilde{\alpha}\ast (D_{a_{w}^{-1}}\circ \tilde{\beta})$. 
These two paths start 
at  $\tilde{x}_{0}$ and end at $w\cdot \tilde{x}$. Since $\widetilde{X^{T}}$ 
is simply--connected, these paths are path-homotopic, i.e.,
$w\cdot \tilde{\beta}\simeq \tilde{\alpha}\ast(D_{a_{w}^{-1}}\circ \tilde\beta)$. 
After composing 
these paths with the projection map $\pi$ we obtain 
$w\cdot \beta \simeq \alpha \ast \beta$ and this 
in turn  implies that $\Phi_{w\cdot \beta}=\Phi_{\beta}\circ \Phi_{\alpha}$. 
On the other hand, recall that the action of $W$ is compatible 
with the isomorphism $\Phi_{\beta}$. This means that 
$\Phi_{w\cdot \beta}=w\Phi_{\beta}(w^{-1})$. This together with the fact that 
$\Phi_{w\cdot \beta}=\Phi_{\beta}\circ \Phi_{\alpha}$ show that the following 
diagram is commutative 
\[
\xymatrix{
\widetilde{T}_{x_{0}}\ar[d]_{\Phi_{\beta}} \ar[r]^{w} &
\widetilde{T}_{x_{0}}\ar[d]_{\Phi_{w\cdot \beta}} \ar[r]^{\Phi_{\alpha}} &
\widetilde{T}_{x_{0}}\ar[d]_{\Phi_{\beta}}\\
\widetilde{T}_{x}\ar[r]_{w} &\widetilde{T}_{x}\ar[r]_{\text{id}} 
&  \widetilde{T}_{x}.
}
\]
In the above diagram the arrows labeled with $w$ represent the maps given 
by the action by  $w$. The commutativity of the previous diagram means precisely 
that $\Phi_{\beta}$ is $\widetilde{W}_{\tilde{x}}$-equivariant. The above shows 
that for every $G$-equivariant slice $U$ there is an isomorphism of $R(G)$-modules
\[
\psi_{U}\colon {}^{P_{U}}K_{G}^{0}(U)\cong R^{\tau_{x}}(T)^{W_{x}}\to 
R(T)^{\widetilde{W}_{\tilde{x}}}\cong H_{\widetilde{W}}^{0}(\tilde{U};\lR).
\]
The isomorphism $\psi_{U}\colon {}^{P_{U}}K_{G}^{0}(U)\to H_{\widetilde{W}}^{0}(\tilde{U};\lR)$ 
constructed above does not depend on the choice of $\tilde{x}\in \tilde{U}$. To see this, 
suppose that $\tilde{x}_{1}\in \tilde{U}$ is another element such that 
$\pi(\tilde{x}_{1})=\pi(\tilde{x})=x$. Let $\tilde{\beta}_{1} \colon [0,1]\to \widetilde{X^{T}}$ be a path in 
$\widetilde{X^{T}}$ from $\tilde{x}_{0}$ to $\tilde{x}_{1}$ and $\beta_{1}=\pi\circ \tilde{\beta}_{1}$. 
Since $\pi \colon \widetilde{X^{T}}\to X^{T}$ 
is the universal cover and $\pi(\tilde{x}_{1})=\pi(\tilde{x})$ we can find a 
unique deck transformation $D$ such that $D(\tilde{x})=\tilde{x}_{1}$. If we identify
$\pi_1(X^T)$ with the group of deck transformations as usual, then the element 
$v:=(D,1)\in \widetilde{W}$ is such that $v\cdot \tilde{x}=\tilde{x}_{1}$.  
As $v\cdot \tilde{x}=\tilde{x}_{1}$ we have 
$\widetilde{W}_{\tilde{x}_{1}}=v\widetilde{W}_{\tilde{x}}v^{-1}$ and 
the action of $v$ on $R(T)$ induces an isomorphism
\[
v \colon R(T)^{\widetilde{W}_{\tilde{x}}}\to R(T)^{\widetilde{W}_{\tilde{x}_{1}}}.
\]
Furthermore we have a commutative diagram
\[
\xymatrix{
{}^{P_{U}}K_{G}^{0}(U)\cong R^{\tau_{x}}(T)^{W_{x}} \ar[r]^{\qquad \Phi_{\beta}}\ar[rd]_{\Phi_{\beta_{1}}} 
& R(T)^{\widetilde{W}_{\tilde{x}}} \ar[r]^{\cong}\ar[d]_{v} &
H_{\widetilde{W}}^{0}(\tilde{U};\lR) \\ 
& R(T)^{\widetilde{W}_{\tilde{x}_{1}}}. \ar[ru]_{\cong}&  
}
\]
This proves that $\psi_{U}$ does not depend on the choice of the element $\tilde{x}$.  
Finally, if $U$ and $V$ are two $G$-equivariant slices with $U\subset V$ it can easily be 
seen that the following diagram commutes 
\begin{equation}\label{commdiag}
\xymatrix{
{}^{P_{V}}K_{G}^{0}(V)\ar[d]\ar[r]^{\psi_{V}} & 
H^{0}_{\widetilde{W}}(\widetilde{V};\lR) \ar[d]\\
{}^{P_{U}}K_{G}^{0}(U)\ar[r]^{\psi_{U}} & 
H^{0}_{\widetilde{W}}(\widetilde{U};\lR). 
}
\end{equation}
In the above diagram the vertical arrows are the corresponding restriction maps.  
The case $q \neq 0$ follows by composing with the natural periodicity isomorphisms.
\end{proof} 

\medskip

Using the previous theorem we can prove the main theorem of this section.

\begin{theorem}\label{ss integer coeff}
Suppose $G$ is a compact Lie group acting on a space $X$ with 
connected maximal rank isotropy subgroups and with a fixed point $x_{0}$. Let 
$p\colon P\to X$ be a $G$-stable $G$-equivariant principal $PU(\H)$-bundle 
that is trivial over $x_{0}$. Assume that $X^{T}$ admits a $W$-equivariant good cover. 
Then there
is a spectral sequence with $E_{2}$-term given by 
\begin{equation*}
E_{2}^{p,q}= \left\{
\begin{array}{ccc}
H^{p}_{\widetilde{W}}(\widetilde{X^{T}};\lR) &\text{ if } &q \text{ is even},\\
0& \text{ if } &q \ \ \text{ is odd}
\end{array}
\right.
\end{equation*}
which converges to ${}^{P}K_{G}^{*}(X)$.
\end{theorem}
\begin{proof}[\bf Proof: ]
Since we are assuming that $X^{T}$ has a $W$-equivariant good cover 
we can find a $G$-equivariant good cover $\U=\{U_i\}_{i\in \I}$ on $X$ by  
Lemma \ref{one-to-one cover}. By \cite[Theorem 22.4.4]{MS}, associated to this cover we have a 
spectral sequence whose $E_{1}^{p,q}$-term is given by equation (\ref{Segalss}). 
For every sequence of elements  $i_{1}\le \dots\le i_{p}$ in $\I$ with 
$U_{i_{1},\dots, i_{p}}$ nonempty 
we have a $G$-homotopy equivalence 
$U_{i_{1},\dots, i_{p}}\simeq G/G_{i_{1},\dots, i_{p}}$. 
This implies 
\[
{}^{P_{i_{1},\dots,i_{p}}}K_{G}^{q}(U_{i_{1},\dots, i_{p}})
\cong {}^{P_{i_{1},\dots,i_{p}}}K_{G_{i_{1},\dots, i_{p}}}^{q}(\{*\})=0
\] 
for $q$ odd. This trivially implies that  $E_{2}^{p,q}=0$ for odd values of $q$. Assume 
now that $q$ is even. Let $\widetilde{\U}$ be the associated $\widetilde{W}$-equivariant 
good cover of 
$\widetilde{X^{T}}$ given by  Lemma \ref{one-to-one cover}.  For any 
sequence $i_{1}\le \dots\le i_{p}$ of elements in $\I$ 
the map $\psi_{U_{i_{1},\dots,i_{p}}}$ constructed in Theorem \ref{identifycontract} provides 
an isomorphism of $R(G)$-modules 
\[
\psi_{U_{i_{1},\dots,i_{p}}}\colon 
{}^{P_{i_{1},\dots,i_{p}}}K_{G}^{q}(U_{i_{1},\dots, i_{p}})
\to H_{\widetilde{W}}^{0}(\widetilde{U}_{i_{1},\dots, i_{p}};\lR).
\]
The naturality of these isomorphism with respect to inclusions between contractible
slices implies that the different maps $\psi_{U_{i_{1},\dots,i_{p}}}$  define an isomorphism 
between the cochain complex $E_{1}^{*,q}$ in the spectral sequence and  
the cochain complex $\{D^{p}_{\widetilde{U}}=
H_{\widetilde{W}}^{0}(\widetilde{U}_{i_{1},\dots, i_{p}};\lR)\}_{p}$ 
whose cohomology computes  $H^{p}_{\widetilde{W}}(\widetilde{X^{T}};\lR)$. 
In particular, this shows that for every even integer $q$ we have an isomorphism of $R(G)$-modules 
 \[
 E_{2}^{p,q}\cong H^{p}_{\widetilde{W}}(\widetilde{X^{T}};\lR). \qedhere
 \]
\end{proof}

\medskip

Next we study the rational coefficients analogue of the spectral sequence studied above 
for a $G$-space with maximal rank isotropy subgroups. 
To start assume that $Y$ is a $W$-CW complex of finite type (we are mainly 
interested in the case where $Y=X^{T}$ for a $G$-CW complex $X$ on which $G$ acts 
with connected maximal rank isotropy). 
Assume that $Y$ has a base point $y_{0}$ that is fixed under the action of $W$. Therefore 
$W$ acts on   $\pi_{1}(Y):=\pi_{1}(Y,y_{0})$ by automorphisms and we can form the semi-direct 
product $\widetilde{W}:=\pi_{1}(Y)\rtimes W$. Let $\pi\colon \widetilde{Y}\to Y$ be the universal 
cover of $Y$. Then $\pi_{1}(Y)$ acts on $\widetilde{Y}$ via deck transformations and as above 
we can lift the action of $W$ on $Y$ to obtain an action of $\widetilde{W}$ on $\widetilde{Y}$.
Fix a $W$-equivariant  homomorphism 
$\phi\colon \pi_{1}(Y)\to \Hom(T,\SS^{1})$. Using this homomorphism we can obtain an 
action of  $\widetilde{W}$ on $R(T)$ defined in the same way as in the case of $Y=X^{T}$.
With this action we can define a new coefficient system  $\lR_{\Q}$ as follows. 

For any compact Lie group $H$ denote by $R(H)_{\Q}$ the complex representation ring of $H$ 
with rational coefficients; that is $R(H)_{\Q}:=R(H)\otimes \Q$. Then the value of $\lR_{\Q}$ 
at an orbit of
the form $\widetilde{W}/\widetilde{W_{i}}$ is defined as
\[
\lR_{\Q}(\widetilde{W}/\widetilde{W_{i}})
=R(T)_{\Q}^{\widetilde{W_{i}}}.
\]
We can regard $\pi_{1}(Y)$ as a normal subgroup of $\widetilde{W}$ and in this way 
we can restrict the coefficient system $\lR_{\Q}$ to obtain 
a coefficient system defined on the orbit category of  $\pi_{1}(Y)$.  

\begin{theorem}\label{Bredon rational}
Suppose that $G$ is a compact connected Lie group. Let $Y$ be a $W$-CW complex 
of finite type. For every $W$-equivariant homomorphism $\phi\colon \pi_{1}(Y)\to \Hom(T,\SS^{1})$ 
there is an isomorphism of $R(G)_{\Q}$-modules 
\[
H^{*}_{\widetilde{W}}(\widetilde{Y};\lR_{\Q})
\cong H_{\pi_{1}(Y)}^{*}(\widetilde{Y};\lR_{\Q})^{W}.
\]
\end{theorem}
\begin{proof}[\bf Proof: ]
To start notice that the coefficient system $\lR_{\Q}$ is induced 
by the $\widetilde{W}$-module $R(T)_{\Q}$. Then
as pointed out in \cite[I. 9]{Bredon} there is an isomorphism of
cochain complexes
\[
C_{\widetilde{W}}^{*}(\widetilde{Y};\lR_{\Q})\cong
\Hom_{\Z[\widetilde{W}]}(C_{*}(\widetilde{Y}),R(T)_{\Q})
\]
and in particular 
$H_{\widetilde{W}}^{*}(\widetilde{Y};\lR_{\Q})\cong
H^{*}(\Hom_{\Z[\widetilde{W}]}(C_{*}(\widetilde{Y}),R(T))\otimes \Q)$. 
Define the cochain complex  
$D^{*}=\Hom(C_{*}(\widetilde{Y}),R(T)_{\Q})$. This cochain complex has a linear
action of $\widetilde{W}$ defined by $(w\cdot f)(x)=wf(w^{-1}x)$. Under this
action we have an isomorphism of cochain complexes
\[
\Hom_{\Z[\widetilde{W}]}(C_{*}(\widetilde{Y}),R(T)_{\Q})= (D^{*})^{\widetilde{W}}
\]
and thus 
$H_{\widetilde{W}}^{*}(\widetilde{Y};\lR_{\Q})
\cong H^{*}( (D^{*})^{\widetilde{W}})$.
On the other hand, define the cochain complex 
$E^{*}=\Hom_{\Z[\pi_{1}(Y)]}(C_{*}(\widetilde{Y}),R(T)_{\Q})$. This cochain 
complex is defined precisely to obtain 
$H^{*}_{\pi_{1}(Y)}(\widetilde{Y};\lR_{\widetilde{W}}\otimes \Q)=H^{*}(E^{*})$. We can 
see $\pi_{1}(Y)$ as a normal subgroup of $\widetilde{W}$ and we have 
$\widetilde{W}/\pi_{1}(Y)=W$. Also, by definition  $(D^{*})^{\pi_1(Y)}=E^{*}$. 
Notice that  we have a natural isomorphism of $R(G)_{\Q}$-modules
\[
(D^{*})^{\widetilde{W}}\cong \left[(D^{*})^{\pi_{1}(Y)} \right]^{W}\cong (E^{*})^{W}.
\]
We conclude that there is a natural isomorphism of $R(G)_{\Q}$-modules
\[
H^{*}_{\widetilde{W}}(\widetilde{Y};\lR_{\Q})
=H^{*}((D^{*})^{\widetilde{W}})\cong H^{*}((E^{*})^{W}).
\]
Consider now $H^{*}(W;E^{*})$; as usual, there are two
spectral sequences computing this group cohomology with coefficients
in a cochain complex. On the one hand, we have
\[
E_{2}^{p,q}=H^{p}(W;H^{q}(E^{*}))
\Longrightarrow H^{p+q}(W;E^{*}).
\]
Since $E^{*}$ is a cochain complex over $\Q$ and $W$ is a finite group 
it follows that
\begin{equation*}
E_{2}^{p,q}=H^{p}(W;H^{q}(E^{*}))\cong \left\{
\begin{array}{ccc}
H^{q}(E^{*})^{W}
&\text{ if } &p=0,\\
0& \text{ if } &p>0 .
\end{array}
\right.
\end{equation*}
On the other hand, we have a spectral sequence
\[
E_{1}^{p,q}=H^{q}(W;E^{p})
\Longrightarrow H^{p+q}(W;E^{*}).
\]
with the differential $d_{1}$ induced by the differential
of the cochain complex $E^{*}$. In this case
\begin{equation*}
H^{q}(W;E^{p})\cong \left\{
\begin{array}{ccc}
(E^{p})^{W}
&\text{ if } &q=0,\\
0& \text{ if } &q>0.
\end{array}
\right.
\end{equation*}
Thus the $E_{2}$-term of this spectral sequence is given by
\begin{equation*}
E_{2}^{p,q}= \left\{
\begin{array}{ccc}
H^{p}((E^{*})^{W})
&\text{ if } &q=0,\\
0& \text{ if } &q>0.
\end{array}
\right.
\end{equation*}
Both of these spectral sequences trivially collapse on the $E_{2}$-term without
extension problems and both converge to $H^{*}(W;E^{*})$.
It follows that there is an isomorphism of $R(G)_{\Q}$ modules
\begin{equation}
H_{\widetilde{W}}^{*}(\widetilde{Y};\lR_{\Q}) 
=H^{*}((E^{*})^{W})
\cong H^{*}(E^{*})^{W}=H^{*}_{\pi_{1}(Y)}(\widetilde{Y};\lR_{\Q})^{W}. \qedhere
\end{equation}
\end{proof}

\medskip

\begin{theorem}\label{ss rational coeff}
Suppose $G$ is a compact Lie group acting on space $X$ with 
connected maximal rank isotropy subgroups and with a fixed
point $x_{0}$. Let 
$p\colon P\to X$ be a $G$-equivariant principal $PU(\H)$-bundle that 
is trivial over $x_{0}$.  Assume that $X^{T}$ is a $W$-CW complex of finite 
type that admits  a 
$W$-equivariant good cover.  Then there is a spectral sequence
with $E_{2}$-term given by  
\begin{equation*}
E_{2}^{p,q}= \left\{
\begin{array}{ccc}
H^{p}_{\pi_{1}(X^{T})}(\widetilde{X^{T}};\lR_{\Q})^{W} 
&\text{ if } &q \text{ is even},\\
0& \text{ if } &q \ \ \text{ is odd}
\end{array}
\right.
\end{equation*}
which converges to ${}^{P}K^{*}_{G}(X)\otimes \Q$.
\end{theorem}
\begin{proof}

Take the $G$-equivariant principal 
$PU(\H)$-bundle $p\colon P\to X$ and let 
$\phi\colon \pi_{1}(X^{T})\to \Hom(T,\SS^{1})$ be the homomorphism provided by 
Proposition \ref{pi1action}. This way we can construct the coefficient system 
$\lR_{\Q}$ as explained above.
Now let $\U=\{U_i\}_{i\in \I}$ be the corresponding  $G$-equivariant good cover on $X$ 
given by Lemma \ref{one-to-one cover}. As above, associated to this 
cover  we have a spectral sequence whose $E_{1}^{p,q}$-term is given by
\[
E_{1}^{p,q}=\prod_{i_{1}\le \dots\le i_{p}}
{}^{P_{i_{1},\dots,i_{p}}}K_{G}^{q}(U_{i_{1},\dots, i_{p}})\otimes \Q
\]
The differential $d_{1}\colon E_{1}^{p,q}\to E_{1}^{p+1,q}$ is given by the alternating sum 
of the  corresponding  restriction maps.  Since $X$ is a $G$-CW complex of finite 
type, this spectral sequence converges to  ${}^{P}K_{G}^{p+q}(X)\otimes \Q$.
When  $q$ is odd, we have $E_{2}^{p,q}=0$ in the same way as with integer coefficients. 
When $q$ is even, proceeding in the same way as in Theorem \ref{ss integer coeff} we see 
that the $E_{2}$-term in this spectral sequence is given by 
\[
E_{2}^{p,q}=H^{p}_{\widetilde{W}}(\widetilde{X^{T}};\lR_{\Q}). 
\]
Using Theorem \ref{Bredon rational} we obtain an isomorphism of $R(G)_{\Q}$-modules. 
\[
H^{*}_{\widetilde{W}}(\widetilde{X^{T}};\lR_{\Q})
\cong H_{\pi_{1}(X^{T})}^{*}(\widetilde{X^{T}};\lR_{\Q})^{W}. \qedhere
\]
\end{proof}

\begin{remark}
Suppose that the cohomology class $\eta_{P}\in H^{3}_{G}(X;\Z)$  associated to $p\colon P\to X$ 
is trivial. In this case the homomorphism $\phi\colon \pi_{1}(X^{T})\to \Hom(T,\SS^{1})$ 
corresponding to $p$ is also trivial. Therefore we have a natural isomorphism 
\[
H^{p}_{\pi_{1}(X^{T})}(\widetilde{X^{T}};\lR_{\Q})^{W} 
\cong H^{p}(X^{T};R(T)_{\Q})^{W} 
\]
and the latter is isomorphic to $H^{*}(X^{T};\Q)\otimes R(G)$ as a module 
over $R(G)_{\Q}$ by \cite[Theorem 4.3]{AG}. In this case 
the spectral sequence collapses at the 
$E_{2}$-term  without extension problems by \cite[Theorem 5.4]{AG}.
\end{remark}

\section{Root systems, actions and cohomology}
\label{Section 4}

In this section we provide some algebraic background that will be essential for our calculations
in twisted equivariant K--theory. We start by setting up some notation. From now on, $G$ will 
denote a compact, simple and simply connected Lie group of rank $r$. Let $\g$ be the Lie algebra 
of $G$ and denote by $\g_{\C}$ its complexification. We fix a maximal torus $T$ in $G$ with Lie 
algebra $\t$ and denote by $W=N_{G}(T)/T$ the corresponding Weyl group.
We are going to see roots as $\C$-linear functions $\alpha\colon \t_{\C}\to \C$. 
The restriction of a root 
to $\t$ is purely imaginary and thus we can also see a root as an $\R$-linear function
$\alpha\colon \t\to i\R$. Let $B(\cdot,\cdot)$ denote the Killing form. Since $G$ is assumed 
to be simple, $B$ is a non-degenerate, negative definite form on $\t$. For each 
root $\alpha$ we can find a unique element $h_{\alpha}\in \t_{\C}$ such that 
$B(H,h_{\alpha})=\alpha(H)$ for every $H\in \t_{\C}$. Define 
$H_{\alpha}=\frac{2h_{\alpha}}{B(h_{\alpha},h_{\alpha})}$ so that $\alpha(H_{\alpha})=2$. 
If we identify $\t_{\C}$ canonically with $\t_{\C}^{**}$ 
then $H_{\alpha}$ corresponds to the 
element $\alpha^{\vee}$ in the inverse root system. 
Consider $\exp_{|\t}\colon \t\to T$, the 
restriction of the exponential map to $\t$ and let 
$\Lambda:=\ker\left(\exp_{|\t}\colon \t\to T\right)$ 
be the unit lattice. For each root $\alpha\in \Phi$ define  
$K_{\alpha}:=2\pi iH_{\alpha}\in \t$. 
Since $G$ is simply connected, the vectors 
$\{ K_{\alpha}\}_{\alpha\in\Phi}$ span the unit lattice by 
\cite[Corollary 1 IX  \S 4 no. 6]{Bourbaki1}. 
Let  $\Delta=\{\alpha_{1},\dots,\alpha_{r}\}$ 
be a fixed set of simple roots and  consider the fundamental weights
$\{\omega_{1},\dots, \omega_{r}\}$ corresponding to the simple roots. 
These fundamental weights are $\R$-linear maps $\omega_{j}\colon \t\to i\R$ such that 
$\omega_{j}(K_{\alpha_{k}})=2\pi i \delta_{kj}$,
where as usual $\delta_{kj}$ denotes the Kronecker delta function and 
$1\le k,j\le r$. Since $G$ is simply connected, the set 
$\{\omega_{1}/2 \pi i,\dots, \omega_{r}/2 \pi i\}$ is a basis for the lattice 
$\Hom(\Lambda,\Z)$. Also, by \cite[Proposition 8.18]{Hall} it follows that 
the vectors $K_{\alpha_{1}},\dots, K_{\alpha_{r}}$ 
form a basis for the integral lattice $\Lambda$. Let 
$\alpha_{0}$ be the highest root. 
We can write $\alpha_{0}$ in the form 
\[
\alpha_{0}= n_{1}\alpha_{1}+n_{2}\alpha_{2}+\cdots+n_{r}\alpha_{r}
\]
for some integers $n_{1},n_{2},\dots,n_{r}\ge 1$. 
The number $h=n_{1}+n_{2}+\cdots+n_{r}+1$ 
is the Coxeter number of the group $G$. On the other hand, 
we can write $K_{\alpha_{0}}$ in the form
\[
K_{\alpha_{0}}=n_{1}^{\vee}K_{\alpha_{1}}+n_{2}^{\vee}K_{\alpha_{2}}
+\cdots+n_{r}^{\vee}K_{\alpha_{r}}
\]
for some integers $n_{1}^{\vee},n_{2}^{\vee},\dots,n_{r}^{\vee}\ge 1$. The number 
$h^{\vee}=n_{1}^{\vee}+n_{2}^{\vee}+\cdots+n_{r}^{\vee}+1$ 
is the dual Coxeter number of the group $G$. 
If  $\rho:=\frac{1}{2}\left(\sum_{\alpha>0}\alpha\right)=\sum_{i=1}^{r}\omega_{i}$, then 
as $\rho(K_{\alpha_{j}})=2\pi i$ for every $1\le j\le r$ we have 
\[
\rho(K_{\alpha_{0}})=2\pi i(n_{1}^{\vee}+n_{2}^{\vee}+\cdots+n_{r}^{\vee})=2\pi i(h^{\vee}-1).
\]
For every $1\le i \le r$ define $d_{i}=d_{\alpha_{i}}:=
\frac{B(h_{\alpha_{0}},h_{\alpha_{0}})}{B(h_{\alpha_{i}},h_{\alpha_{i}})}$. It can 
be seen that each $d_{i}$ is an integer such that $n_{i}=d_{i}n_{i}^{\vee}$. 
Let $\mathfrak{A}(\Delta)$ be the (closed) Weyl alcove that is contained in the 
(closed) Weyl chamber $\mathfrak{C}(\Delta)$ determined by $\Delta$ and that 
contains $0\in \g$. The alcove $\mathfrak{A}(\Delta)$ is bounded by the hyperplanes with 
equations $\alpha_{j}=0$ for $1\le j\le r$ and $\alpha_{0}=2\pi i$. 
As a topological space $\mathfrak{A}(\Delta)$ is an $r$-simplex whose vertices will be labeled 
by $v_{0}, v_{1},\dots,v_{r}$ in such 
a way that $v_{0}=0$ and $v_{j}$ is the vertex that does not lie in the hyperplane 
$\alpha_{j}=0$ for $1\le j\le r$.  For every $1\le j\le r$  we have  
$\alpha_{0}(v_{j})=2\pi i$ so that
\[
\alpha_{j}(v_{j})=\frac{2\pi i}{n_{j}} \text{ for all } 1\le j\le r
\]
and $\alpha_{j}(v_{k})=0$ for $1\le k\le r$ with $k\ne j$. 
On the other hand, if $f$ is an element in the weight lattice we denote by 
$\theta_{f}\colon T\to \SS^{1}$ the unique homomorphism of Lie groups whose derivative is $f$. 
Notice that the lattice $\Hom(\Lambda,\Z)$ can be identified with  $\Pi=\Hom(T,\SS^{1})$ under 
the following assignment
\begin{align*}
\Gamma\colon \Hom(\Lambda,\Z)&\to \Hom(T,\SS^{1})\\
\beta&\mapsto \Gamma_{\beta},
\end{align*}
where  $\Gamma_{\beta}\colon T\to \SS^{1}$ is the unique 
homomorphism that satisfies the equation 
\[
\Gamma_{\beta}(\exp(tK_{\alpha}))=e^{2\pi i t\beta(K_{\alpha})}
\]
for every real number $t\in \R$ and every root $\alpha\in \Phi$.  
From now on we will identify $\Hom(\Lambda,\Z)$ 
and $\Hom(T,\SS^{1})$ using this isomorphism. 

\medskip

When $G$ is compact, simple and simply connected acting on itself by
conjugation, we have $H^3_G(G;\Z) \cong \Z$. 
By \cite[Proposition 3.1]{Meinrenken} it follows that we can choose 
a generator $\eta\in H^{3}_{G}(G;\Z)$ in such a way that the restriction 
of $\eta$ to $T$ corresponds to restriction of the basic inner 
product on $\g$ to $\Lambda$. We fix this choice from now on. 
Let $p_n \colon P_n \to G$ be a $G$-equivariant
principal $PU(\H)$-principal bundle that corresponds to the class $n\eta$.
By Proposition \ref{pi1action}, we have an associated homomorphism $\phi_n \colon \Lambda \to \Pi$. 
The following proposition provides an explicit expression for $\phi_n$ in
terms of the geometry of $G$.

\begin{proposition}\label{Inducedhom}
Let $G$ be a compact, simple and simply connected Lie group and $n\in \Z$ an integer. 
The homomorphism $\phi_{n}\colon \Lambda\to \Pi$ corresponding to $p_{n}\colon P_{n}\to G$  
is given by
\[
\phi_{n}(K_{\alpha})=\theta_{\alpha}^{nd_{\alpha}}
\]
for every $\alpha\in \Phi$. Here $d_{\alpha}$ is the integer defined by
\[
d_{\alpha}=\frac{B(h_{\alpha_{0}},h_{\alpha_{0}})}{B(h_{\alpha},h_{\alpha})}. 
\]
\end{proposition}
\begin{proof}[\bf Proof: ]
Let us show first the result for $n=1$; that is, let us prove the result   
for a principal bundle $p\colon P\to G$ whose class corresponds to 
$\eta\in H_{G}^{3}(G;\Z)$. The class $\eta$ is chosen so that 
its restriction to $T$ corresponds to restriction of the basic inner 
product on $\g$ to $\Lambda$. The restriction of this basic inner product  
defines a bilinear form 
$\left<\cdot,\cdot\right>\colon \Lambda\times \Lambda\to \Z$ in such a way that  
the smallest non-zero element in $\Lambda$ has length $\sqrt{2}$. With our conventions 
the smallest element in $\Lambda$ is precisely $K_{\alpha_{0}}$, 
where as above $\alpha_{0}$ is the  highest root. 
Therefore $\left<K_{\alpha_{0}},K_{\alpha_{0}}\right>=2$. On the other hand, 
since $\left<\cdot,\cdot\right>$ is a non-degenerate, $W$-invariant form on $\t$
we have that $\left<\cdot,\cdot\right>=cB(\cdot, \cdot)$ for a certain constant $c$. 
This is because the Killing form is, up to multiplication by a constant, 
the only $W$-invariant non-degenerate form on $\t$. Since 
$\left<K_{\alpha_{0}},K_{\alpha_{0}}\right>=2$ we conclude that 
\[
c=\frac{B(h_{\alpha_{0}},h_{\alpha_{0}})}{2(2\pi i)^{2}}.
\]
Therefore for every $X\in \t$ and every root $\alpha$ we have 
\begin{align*}
\left<X,K_{\alpha}\right>&=\frac{B(h_{\alpha_{0}},h_{\alpha_{0}})}{2(2\pi i)^{2}}B(X,K_{\alpha})
=\frac{B(h_{\alpha_{0}},h_{\alpha_{0}})}{2(2\pi i)}B(X,H_{\alpha})\\
&=\frac{B(h_{\alpha_{0}},h_{\alpha_{0}})}{B(h_{\alpha},h_{\alpha})}
\left(\frac{{\alpha}(X)}{2\pi i}\right)=d_{\alpha}\left(\frac{\alpha(X)}{2\pi i}\right).
\end{align*}
This shows that the homomorphism $\Lambda\to \Z$ given by 
$X\mapsto \left<X,K_{\alpha}\right>$ 
agrees with $\alpha/2\pi i$. With the identification of 
$\Hom(\Lambda,\Z)$ and $\Hom(T,\SS^{1})$ given above this homomorphism 
corresponds precisely to $\theta_{\alpha}^{d_{\alpha}}$. 
This proves that $\phi_{1}(K_{\alpha})=\theta_{\alpha}^{d_{\alpha}}$ for 
every root $\alpha\in \Phi$.

In general, if $p_{n}\colon P_{n}\to G$ is a principal bundle whose class 
corresponds to $n\eta\in H_{G}^{3}(G;\Z)$, then the restriction to $T$ 
of $n\eta$ corresponds to the restriction of $n\left<\cdot,\cdot\right>$ to $\Lambda$,
hence the result holds because of the case $n=1$.
\end{proof}

\medskip

In the next section, it will be important to understand the structure of 
$H^{*}_{\Lambda}(\t,R(T)_{\Q})$ as a $W$-module, hence we give a description
here in terms of the geometry of the group. For this, let $R(T)^{sgn}$ 
denote the representation ring $R(T)$ with $W$-action given as follows: 
if $w\in W$ and $p\in R(T)^{sgn}$, then  
\[
w\bullet p=(-1)^{\ell(w)}w\cdot p.
\]
Here $\ell(w)$ denotes the length of $w$ relative to the reflections 
$s_{\alpha_{1}},\dots, s_{\alpha_{r}}$ and $w\cdot p$ denotes the natural action of 
$W$ on $R(T)$. Also, we denote $R(T)^{sgn}_{\Q}=R(T)^{sgn}\otimes \Q$.

On the other hand, let $R:=\Z[\Lambda]$ and $M:=\Z[\Pi]=R(T)$. We can see $M$ as a module over 
$R$ via the homomorphism $\phi_{n}\colon \Lambda\to \Pi$. This homomorphism is injective 
by \cite[Proposition 3.1]{Meinrenken} whenever $n\ne 0$. The lattices
$\Lambda$ and $\Pi$ have the same rank so that $M$ is a finitely 
generated $R$-module. The elements $K_{\alpha_{1}},\dots, K_{\alpha_{r}}$ 
form a free basis for the lattice $\Lambda$. If $x_{i}:=K_{\alpha_{i}}-1$ for $1\le i \le r$  
then  the Koszul complex $K_{*}=K(x_{1},\dots, x_{r})$ forms a free resolution of $\Z$ seen as a 
trivial module over $R$. Recall that in $K_{*}$ we have generators $a_{1},\dots, a_{r}$ of degree 
one and for each $0\le p\le r$  the module  $K_{p}$ is a free $R$-module generated by elements of 
the form $a_{i_{1},\dots, i_{p}}:=a_{i_{1}}\cdots a_{i_{p}}$, where $i_{1},\dots, i_{p}$ 
runs through all the sequences of integers satisfying 
$1\le i_{1}<\cdots <i_{p}\le r$. The differential $\partial\colon K_{p}\to K_{p-1}$ 
is given by the formula 
\[
\partial(a_{i_{1},\dots, i_{p}})=
\sum_{s=1}^{p}(-1)^{s-1}x_{i_{s}}a_{i_{1},\dots, \hat{i}_{s},\dots, i_{p}}
\]
Let $J_{n}$ denote the ideal in $R(T)^{sgn}$ generated by 
$\theta_{\alpha_{i}}^{nd_{i}}-1$ for $1\le i\le r$.

\begin{theorem}\label{ideal-J}
Suppose that $G$ is a compact, simple and simply connected Lie group of rank equal to $r$,
and $n\ne 0$ is an integer. 
If $p\ne r$, we have $H^{p}_{\Lambda}(\t,R(T))=0$ and for $p=r$ 
there is a $W$-equivariant isomorphism 
\[
H^{r}_{\Lambda}(\t,R(T))\cong R(T)^{sgn}/J_{n}.
\]
Similarly, with rational coefficients we have $H^{p}_{\Lambda}(\t,R(T)_{\Q})=0$ for $p\ne r$
and for $p=r$, there is  a $W$-equivariant isomorphism  
\[
H^{r}_{\Lambda}(\t,R(T)_{\Q})\cong R(T)^{sgn}_{\Q}/J_{n}.
\]
\end{theorem}
\begin{proof}[\bf Proof: ]
Recall that the elements $K_{\alpha_{1}},\dots ,K_{\alpha_{r}}$ form a basis for the integral 
lattice $\Lambda$. Therefore $\t$ has the structure of a $\Lambda$-CW complex in such a way 
that for each sequence of integers $1\le i_{1}<\cdots <i_{p}\le r$ we have a 
$p$-dimensional $\Lambda$-cell $e_{i_{1},...,i_{p}}$ corresponding to the sequence 
$\{K_{\alpha_{i_{1}}},\dots,K_{\alpha_{i_{p}}}\}\subset \Lambda$.
If we consider the cellular complex $C_{*}(\t)$, then 
the natural action of $W$ on $\t$ makes $C_{*}(\t)$ into a free resolution of $\Z$ as a 
trivial module over $R = \Z[\Lambda]$, where all the maps in sight are $W$-equivariant. 
Moreover, if we identify the element  $e_{i_{1},...,i_{p}}\in C_{p}(\t)$ with 
$a_{i_{1},..,i_{p}}$ we obtain an isomorphism of $C_{*}(\t)$ with the Koszul complex. 
Via this isomorphism we can give $K_{*}$ the structure of a $W$-equivariant chain complex. 
By definition 
\[
H^{*}_{\Lambda}(\t;R(T))=H^{*}(\Hom_{\Z[\Lambda]}(C_{*}(\t),R(T)))\cong H^{*}(\Hom_{R}(K_{*},M)).
\]
Let $K^{*}$ be the dual of the Koszul complex. Explicitly, 
$K^{p}=\bigwedge^{p}(R^{r})$
and the differential $\delta_{x}\colon \wedge^{p}(R^{r})\to \wedge^{p+1}(R^{r})$ is given by 
$\delta_{x}(y)=x\wedge y$, where $x=(x_{1},\dots, x_{r})\in R^{r}$. 
We have an isomorphism of cochain complexes 
$\Hom_{R}(K_{*},M)\cong K^{*}\otimes_{R} M$. Also, 
the elements $x_{1},\dots, x_{r}$ form a regular sequence on $M$ so by 
\cite[Corollary 17.5]{Eisenbud} it follows that for $p\ne r$ 
\[
H^{p}_{\Lambda}(\t;R(T))\cong  H^{*}(K^{*}\otimes_{R} M)=0.
\] 
Moreover, for $p=r$ we have an isomorphism of $R$-modules
\[
H^{r}_{\Lambda}(\t;R(T))\cong  H^{r}(K^{*}\otimes_{R} M)=M/(x_{1},\dots,x_{r})M.
\]
By Proposition \ref{Inducedhom} we have  
$\phi_{n}(K_{\alpha_{i}})=\theta_{\alpha_{i}}^{nd_{i}}$ and thus 
we have an isomorphism of  $R$-modules 
\[
H^{r}_{\Lambda}(\t;R(T))\cong M/(x_{1},\dots,x_{r})M\cong 
R(T)/(\theta_{\alpha_{1}}^{nd_{1}}-1,\dots,\theta_{\alpha_{r}}^{nd_{r}}-1).
\] 
However, the above isomorphism is not  $W$-equivariant. To obtain the correct 
$W$-action on $H^{r}_{\Lambda}(\t;R(T))$ we observe that we have a $W$-equivariant 
isomorphism 
\[
\Hom_{\Z[\Lambda]}(C_{r}(\t),R(T))\cong R(T)^{sgn}.
\] 
With this observation we can conclude that we have a $W$-equivariant isomorphism 
\[
H^{r}_{\Lambda}(\t;R(T))\cong  H^{r}(K^{*}\otimes_{R} M)=R(T)^{sgn}/J_{n}.
\]
The rational coefficients case is proved in a similar way.
\end{proof}

\section{The twisted equivariant K-theory of inertia spaces} 

Suppose that $X$ is a compact $G$-CW complex in such a way that 
the action of $G$ on $X$ has connected maximal rank isotropy subgroups. 
Consider the inertia space 
\[
\Lambda X:=\left\{(x,g) | g\cdot x=x \right\}.
\]
The group $G$ acts on $\Lambda X$ by the assignment $h\cdot (x,g)=(h\cdot x,hgh^{-1})$. 
Notice that the isotropy subgroup of this action at $(x,g)$ is the centralizer of $g$ in $G_{x}$. 
If we assume that  $G_{x}$ has torsion free fundamental group for every $x\in X$, 
then the action of $G$ on $\Lambda X$ also has connected maximal rank isotropy subgroups 
(see \cite[Section 2]{AG}). Assume also that $x_{0}\in X$ is a point 
fixed by the $G$-action, then $(x_{0},1_{G})\in \Lambda X$  is also fixed by the 
$G$-action. We work under these assumptions from now on.

Since $G$ is compact and simply connected we have 
$H_{G}^{3}(G;\Z)\cong \Z$. We fix a generator $\eta\in H_{G}^{3}(G;\Z)$  
as in the previous section and suppose that  $p_{n}\colon P_{n}\to G$ is  
a $G$-equivariant principal $PU(\H)$-bundle that corresponds 
to the class $n\eta\in H^{3}_{G}(G;\Z)$, where $n\in \Z-\{0\}$.  Since $G$ 
is assumed to be simply connected the restriction of the bundle $p_{n}\colon P_{n}\to G$  
to $1_{G}$ is trivial. Recall that the bundle 
$p_{n}\colon P_{n}\to G$ induces a  homomorphism 
$\phi_{n}=\phi_{P_{n}}\colon \Lambda \to \Pi=\Hom(T,\SS^{1})$ that is $W$-equivariant by 
Proposition \ref{pi1action}. Denote by $\pi_{2}\colon \Lambda X \to G$ the projection onto the 
second component and let $q_{n}\colon Q_{n}\to \Lambda X$ be the pullback 
of $P_{n}$ under $\pi_{2}$. We are interested in computing 
${}^{Q_{n}}K_{G}^{*}(\Lambda X)\otimes \Q$ as a module over $R(G)_{\Q}$. 

As a first step notice that $(\Lambda X)^{T}=X^{T}\times T$ and thus 
$\widetilde{(\Lambda X)^{T}}=\widetilde{X^{T}}\times \t$.  
Also, $\pi_{1}((\Lambda X)^{T})=\pi_{1}(X^{T})\times \Lambda$
and the homomorphism associated to the bundle $q_{n}\colon Q_{n}\to \Lambda X$ is the map 
\[
\phi_{Q_{n }}=\phi_{n}\circ \pi_{2}\colon \pi_{1}(X^{T})\times \Lambda\to \Hom(T,\SS^{1}).
\]
We can see $R(T)$ as a module over $\pi_{1}(X^{T})\times \Lambda$ via this homomorphism. 
To compute ${}^{Q_n}K_{G}^{*}(\Lambda X)\otimes \Q$ we are going 
to use the spectral sequence previously described. By Theorem \ref{ss rational coeff}, 
the $E_{2}$-term in this spectral sequence is given by 
\begin{equation}\label{E2-term1}
E_{2}^{p,q}= \left\{
\begin{array}{ccc}
H^{p}_{\pi_{1}(X^{T})\times \Lambda}(\widetilde{X^{T}}\times \t;R(T)_{\Q})^{W} 
&\text{ if } &q \text{ is even},\\
0& \text{ if } &q \ \ \text{ is odd.}
\end{array}
\right.
\end{equation}
Next we identify  $H^{p}_{\pi_{1}(X^{T})\times \Lambda}(\widetilde{X^{T}}\times \t;R(T)_{\Q})$. 
Note that as a  module over $\pi_{1}(X^{T})\times \Lambda$ there is an isomorphism 
$R(T)_{\Q} \cong \Q\otimes R(T)_{\Q}$, where on the right hand side, $\pi_{1}(X^{T})\times \Lambda$ 
acts trivially on $\Q$ and $\pi_{1}(X^{T})\times \Lambda$ acts on $R(T)_{\Q}$ by the 
homomorphism $\phi_{Q_{n}}$. 
Using a suitable version of the K\"unneth theorem and applying Theorem \ref{ideal-J}, we obtain a 
$W$-equivariant isomorphism 
\begin{align*}
&H^{n}_{\pi_{1}(X^{T})\times \Lambda}(\widetilde{X^{T}}\times \t;R(T)_{\Q})
\cong H^{n-r}_{\pi_{1}(X^{T})}(\widetilde{X^{T}};\Q)\otimes H^{r}_{\Lambda}(\t;R(T)_{\Q})\\
\cong & H^{n-r}(X^{T};\Q)\otimes H^{r}_{\Lambda}(\t;R(T)_{\Q})\cong 
H^{n-r}(X^{T};\Q)\otimes (R(T)^{sgn}_{\Q}/J_{n}).
\end{align*}
The above can be summarized in the following theorem.

\begin{theorem}\label{TheoremE2-term}
Let $G$ be a compact, simple and simply connected Lie group of rank equal to 
$r$ and $n\ne 0$  an integer. Suppose 
that $X$ is a compact $G$-CW complex such that $G_{x}$ is a connected subgroup of maximal 
rank that has torsion free fundamental group for every $x\in X$ and that there 
is a point fixed by the action of $G$. 
Then there is a spectral sequence with $E_{2}$-term given by 
\begin{equation*}
E_{2}^{p+r,q}= \left\{
\begin{array}{ccc}
\left[ H^{p}(X^{T};\Q)\otimes (R(T)^{sgn}_{\Q}/J_{n}) \right]^{W}
&\text{ if } &p\ge 0\text{ and } q \text{ is even},\\
0& \text{ if } &q \ \ \text{ is odd}
\end{array}
\right.
\end{equation*} 
which converges to ${}^{Q_{n}}K_{G}^{*}(\Lambda X)\otimes \Q$.
\end{theorem}

\begin{remark}
It seems reasonable to conjecture 
that with the given hypotheses the above spectral sequence should always
collapse at the $E_{2}$-term. We will provide some key examples where 
this collapse can be verified.
\end{remark}

\subsection{$G$ acting on itself by conjugation}

To begin we consider the case $X=\{*\}$ with the trivial $G$-action.  
If  $G$ is a compact, simple and simply connected Lie group then the hypotheses of 
Theorem \ref{TheoremE2-term} are satisfied trivially. Here the inertia 
space $\Lambda X$ is $G$-homeomorphic to $G$, where $G$ acts on itself by conjugation.
Let $n\ne 0$ and choose 
$p_{n}\colon P_{n}\to G$  a $G$-equivariant principal $PU(\H)$-bundle that corresponds 
to the cohomology class $n\eta\in H^{3}_{G}(G;\Z)$. In this case the bundle 
$Q_{n}=\pi_{2}^{*}(P_{n})$ agrees trivially with $P_{n}$. Also $H^{p}(X^{T};\Q)$ is trivial for 
$p\ne 0$ and $H^{0}(X^{T};\Q)\cong \Q$ with the trivial $W$-action. By Theorem \ref{TheoremE2-term}, 
the $E_{2}$-term in the spectral sequence computing 
${}^{P_{n}}K^{*}_{G}(G)\otimes \Q$ is such that 
\[
E_{2}^{r,q}= (R(T)^{sgn}_{\Q}/J_{n})^{W}
\]
for even values of $q$ and $E_{2}^{p,q}=0$ in all other cases. 
Consider the short exact sequence of $W$-modules
\[
0\to J_{n} \to  R(T)^{sgn}_{\Q} 
\to R(T)^{sgn}_{\Q}/J_{n}\to 0.
\]
Since we are working in characteristic zero and $W$ is a 
finite group, the exactness 
of this sequence is preserved at the level of $W$-invariants; that is, 
there is a short exact sequence
\[
0\to J^{W}_{n} \to  (R(T)^{sgn}_{\Q})^{W} 
\to (R(T)^{sgn}_{\Q}/J_{n})^{W}\to 0
\]
and so $(R(T)^{sgn}_{\Q}/J_{n})^{W}\cong (R(T)^{sgn}_{\Q})^{W}/J_{n}^{W}$. 
Consider the element  
$\delta:=\theta_{\rho}^{-1}\prod_{ \alpha\in \Phi^{+}}(\theta_{\alpha}-1)$,
where $\rho$ denotes the half sum of the positive roots. 
Notice that $\delta$ is the  denominator in the Weyl character formula. 
By \cite[Proposition 2, VI \S 3 no. 3]{Bourbaki} 
we have that $\delta\in (R(T)^{sgn}_{\Q})^{W}$
and the map 
\begin{align*}
\Psi\colon  (R(T)^{sgn}_{\Q})^{W}&\to R(T)_{\Q}^{W}\cong R(G)_{\Q} \\
p&\mapsto \frac{p}{\delta}
\end{align*}
is an isomorphism of $R(G)_{\Q}$-modules. 
Let $L_{n}=\Psi(J_{n}^{W})\subset R(G)_{\Q}$. 
The map $\Psi$ 
defines an isomorphism of $R(G)_{\Q}$-modules 
\begin{equation*}
(R(T)^{sgn}_{\Q})^{W}/J_{n}^{W}\cong R(G)_{\Q}/L_{n}.
\end{equation*}
Therefore the $E_{2}$-term of the spectral sequence 
is given by 
\begin{equation*}
E_{2}^{p,q}= \left\{
\begin{array}{ccc}
R(G)_{\Q}/L_{n}
&\text{ if } &p=r \text{ and } q \text{ is even},\\
0& \text{ if } &q \ \ \text{ in other cases}
\end{array}
\right.
\end{equation*} 
as a module over $R(G)_{\Q}$. In this situation trivially there are no extension problems and we 
conclude that there is an isomorphism of  $R(G)_{\Q}$-modules 
\begin{equation*}
{}^{P_{n}}K^{p}_{G}(G)\otimes \Q\cong \left\{
\begin{array}{ccl}
R(G)_{\Q}/L_{n} &\text{ if } & p\equiv r\ (\text{mod } 2),\\
0 &\text{ if } & p\equiv r+1\ (\text{mod } 2).
\end{array}
\right.
\end{equation*}
Suppose now that $k\ge 0$ is an integer such that $n=k+h^{\vee}$; that is, $k=n-h^{\vee}$.
With this further assumption, we show in Proposition \ref{Verlinde} that 
$L_{n}$ is precisely the rational Verlinde ideal $I_{k}$ at level $k=n-h^{\vee}$ and so 
$R(G)_{\Q}/L_{n}$ is the rational Verlinde algebra $V_{k}(G)_{\Q}=R(G)_{\Q} /I_{k}$ at level 
$k=n-h^{\vee}$. Therefore for $p\equiv r\ (\text{mod } 2)$ we have that 
${}^{P_{n}}K^{p}_{G}(G)\otimes \Q$ is isomorphic as a module over $R(G)_{\Q}$ to 
the rational Verlinde algebra $V_{k}(G)_{\Q}=R(G)_{\Q} /I_{k}$ at level $k=n-h^{\vee}$. 
This shows that our computations in this particular case agree with the 
celebrated theorem of Freed, Hopkins, Teleman (see \cite[Theorem 1]{FHTIII}).  
 
\medskip

\begin{example} Suppose that $G=SU(m)$. We provide here an explicit basis for the 
$\Q$-vector space $(R(T)^{sgn}_{\Q}/J_{n})^{W}\cong R(G)_{\Q}/L_{n}$ in this case. 
Let $T\cong (\SS^{1})^{m-1}$ be the maximal torus  consisting of all diagonal 
matrices in $SU(m)$. In particular, we have $r=m-1$. The Weyl group 
$W=\Sigma_{m}$ acts by permuting the diagonal entries of the elements in 
$T$. We  can choose $\Delta=\{\alpha_{1}=X_{1}-X_{2},\dots, \alpha_{m-1}=X_{m-1}-X_{m}\}$ 
as a set of simple roots. In this case we have
\[
R(T)^{sgn}_{\Q}=\Q[x_{1},\dots, x_{m}]/(x_{1}x_{2}\cdots x_{m}=1),
\]
where $\Sigma_{m}$ acts by signed permutations. More precisely, if $\sigma\in \Sigma_{m}$
then 
\[
\sigma\bullet (x_{i_{1}}^{a_{1}}\cdots x_{i_{k}}^{a_{k}})
=(-1)^{sgn(\sigma)}x_{\sigma^{-1}(i_{1})}^{a_{1}}\cdots x_{\sigma^{-1}(i_{k})}^{a_{k}}.
\]
Suppose that $n\ge m$ is a fixed integer. Then $J_{n}$ is the ideal in 
$R(T)^{sgn}_{\Q}$ generated 
by 
\[
\theta_{\alpha_{i}}^{n}-1=\left(\frac{x_{i}}{x_{i+1}}\right)^{n}-1, \text { for } 1\le i\le m-1.
\]
We conclude that
\[
R(T)^{sgn}_{\Q}/J_{n}=\Q[x_{1},\dots, x_{m}]/
(x_{1}x_{2}\cdots x_{m}=1, x_{1}^{n}=x_{2}^{n}=\cdots =x_{m}^{n}),
\]
with $\Sigma_{m}$ acting by signed permutations. 
Let $\A\colon \Q[x_{1},\dots, x_{m}]\to \Q[x_{1},\dots, x_{m}]$ 
denote the anti-symmetrization operator, that is, 
\[
\A(p(x_{1},\dots,x_{m}))=\sum_{\sigma\in \Sigma_{m}}(-1)^{sgn(\sigma)}
p(x_{\sigma(1)},\dots,x_{\sigma(m)}).
\]
In this example the set 
$\{\A(x_{1}^{i_{1}}\cdots x_{m-1}^{i_{m-1}})\}_{1\le i_{m-1}<i_{m-2}<\cdots<i_{1}\le n-1}$ 
forms a basis as a $\Q$-vector space for $(R(T)^{sgn}_{\Q}/J_{n})^{\Sigma_{m}}$.   
Thus as a $\Q$-vector space
\begin{equation*}
{}^{P_{n}}K^{q}_{SU(m)}(SU(m))\otimes \Q\cong \left\{
\begin{array}{lcl}
\Q^{ \binom {n-1} {m-1}} &\text{ if } & q\equiv m-1\ (\text{mod } 2),\\
0 &\text{ if } & q\equiv m\ (\text{mod } 2).
\end{array}
\right.
\end{equation*}
One can compare this computation with the general results obtained in \cite[Theorem 1.1]{Douglas1} 
and also with the well known description of the Verlinde algebra in terms of weights, 
see \cite[Section 5.3]{Meinrenken}
for instance.

\end{example}

\subsection{Inertia sphere}\label{inertia sphere}
Suppose now that $X=\SS^{\g}$ is the one point 
compactification of the 
Lie algebra $\g$. The group $G$ acts on $\SS^{\g}$ via the adjoint representation and fixing the 
point at infinity. This action has connected 
maximal rank isotropy subgroups. (See \cite[Section 6.2]{AG}). Moreover, when 
$G$ is simple and simply connected, the isotropy subgroup $G_v$ of every $v\in \g$ 
has torsion--free fundamental group. To see this, we may assume without loss 
of generality that $v$ belongs to the Weyl chamber $\mathfrak{C}(\Delta)$. Let $I\subset \Delta$ 
be the set of simple roots for which $\alpha(v)=0$. Then $G_{v}=G_{I}$ is the maximal 
rank subgroup of $G$ that contains $T$ and that has $I$ as a set of simple roots. Let 
$\Lambda_{I}$ be the lattice generated by the set $\{K_{\alpha}\}_{\alpha\in I}$. By  
\cite[Theorem 7.1]{Dieck}, we have that $\pi_{1}(G_{I})=\Lambda/\Lambda_{I}$ is a free 
abelian group of rank $|\Delta-I|$. We conclude that the action of $G$ on 
$\Lambda \SS^{\g}$ has connected maximal rank isotropy subgroups and the hypotheses of 
Theorem \ref{TheoremE2-term} are satisfied.  In this example 
$X^{T}=\SS^{\t}$. Notice that $H^{0}(\SS^{\t};\Q)\cong \Q$ with $W$ acting 
trivially,  $H^{r}(\SS^{\t};\Q)\cong \Q^{sgn}$, where $w\cdot x=(-1)^{\ell(w)}x$ for all 
$w\in W$ and all $x\in \Q$ and  $H^{p}(\SS^{\t};\Q)=0$ for all $p\ne 0,r$. It follows by 
Theorem \ref{TheoremE2-term}  that the $E_{2}$-term of the spectral sequence computing 
${}^{Q_{n}}K_{G}^{*}(\Lambda \SS^{\g})\otimes \Q$ is given as follows: for $q$ even 
\[
E_{2}^{r,q}= [H^{0}(X^{T};\Q)\otimes (R(T)^{sgn}_{\Q}/J_{n})]^{W}
\cong (R(T)^{sgn}_{\Q})^{W}/J_{n}^{W}\cong R(G)_{\Q}/L_{n}.
\]
Also, 
\[
E_{2}^{2r,q}=[H^{r}(X^{T};\Q)\otimes (R(T)^{sgn}_{\Q}/J_{n})]^{W}
\cong (\Q^{sgn}\otimes (R(T)^{sgn}_{\Q}/J_{n}))^{W}.
\]
Notice that $\Q^{sgn}\otimes R(T)^{sgn}_{\Q}\cong R(T)_{\Q}$ as a $W$-module, where on the 
right hand side $W$ acts by the natural action. Let $K_{n}$ denote the ideal in $R(T)_{\Q}$ 
generated by $\theta_{\alpha_{i}}^{nd_{i}}-1$ 
for $1\le i\le r$, where $W$ acts on $K_{n}$ by the natural action. Note that $K_{n}$ and $J_{n}$ 
are equal as sets, but $W$ acts on $J_{n}$ by signed permutations and on $K_{n}$ by permutations. 
It follows that we have an isomorphism of $R(G)_{\Q}$-modules 
\[
E_{2}^{2r,q}\cong (R(T)_{\Q}/K_{n})^{W}\cong R(G)_{\Q}/K_{n}^{W}.
\]
For other values of $p$ and $q$ we have $E_{2}^{p,q}=0$. 
In this case the only possibly non-trivial 
differential is $d_{r-1}$ for odd values of $r\ge 3$. However, 
comparing this spectral sequence with the one computing ${}^{P_{n}}K_{G}^{*}(G)$ 
via the inclusion map $i\colon *\to G$ and the projection map $\pi\colon G\to *$, we conclude  
that such differentials must be trivial. Hence
\begin{equation*}
E_{\infty}^{p,q}\cong \left\{
\begin{array}{cl}
R(G)_{\Q}/L_{n}
&\text{ if } q \text{ is even and } p=r,\\
R(G)_{\Q}/K_{n}^{W}
&\text{ if } q \text{ is even and } p=2r,\\
0& \text{ otherwise, } 
\end{array}
\right.
\end{equation*}
as a module over $R(G)_{\Q}$.

When $r$ is odd, then trivially there are no extension problems.
When $r$ is even, by comparing the spectral sequence computing 
${}^{Q_{n}}K_{G}^{p}(\Lambda \SS^{\g})\otimes \Q$ with that computing 
${}^{P_{n}}K_{G}^{p}(G)\otimes \Q$, we conclude that there are no 
extension problems in this case either. As a corollary, the following is obtained.

\begin{corollary}
Suppose that $G$ is a compact, simple and simply connected Lie group. 
Let  $n\ne 0$ be an integer and let $r$ be the rank of $G$. Then we have
the following isomorphisms $R(G)_{\Q}$-modules:
\begin{enumerate}
\item For odd values of $r$ 
\begin{equation*}
{}^{Q_{n}}K_{G}^{p}(\Lambda \SS^{\g})\otimes \Q\cong \left\{
\begin{array}{cl}
R(G)_{\Q}/K_{n}^{W}
&\text{ for  } p \text{ even,} \\
R(G)_{\Q}/L_{n}
&\text{ for  } p \text{ odd.}
\end{array}
\right.
\end{equation*}

\item For even values of $r$ 
\begin{equation*}
{}^{Q_{n}}K_{G}^{p}(\Lambda \SS^{\g})\otimes \Q\cong \left\{
\begin{array}{cl}
R(G)_{\Q}/K_{n}^{W}\oplus R(G)_{\Q}/L_{n}
&\text{ for  } p \text{ even,} \\
0
&\text{ for  } p \text{ odd.}
\end{array}
\right.
\end{equation*}
\end{enumerate}
\end{corollary}

As before, if we further assume that $k\ge 0$ is such that $k=n-h^{\vee}$ 
then we can identify $R(G)_{\Q}/L_{n}$ with the Verlinde algebra 
$V_{k}(G)_{\Q}=R(G)_{\Q} /I_{k}$ at level $k=n-h^{\vee}$. 

\medskip

As a particular example let us suppose that $G=SU(m)$ and 
that $n\ge m$. We already know that in this case 
\[
R(T)^{sgn}_{\Q}/J_{n}=\Q[x_{1},\dots, x_{m}]/
(x_{1}x_{2}\cdots x_{m}=1, x_{1}^{n}=x_{2}^{n}=\cdots x_{m}^{n}),
\]
with $W=\Sigma_{m}$ acting by {\bf{signed permutations}}. Also, 
 $\{\A(x_{1}^{i_{1}}\cdots x_{m-1}^{i_{m-1}})\}_{1\le i_{1}<i_{2}<\cdots<i_{m-1}\le n-1}$ 
forms a basis for $(R(T)^{sgn}_{\Q}/J_{n})^{\Sigma_{m}}\cong R(G)_{\Q}/L_{n} $ 
as a $\Q$-vector space.  On the other hand, 
\[
R(T)_{\Q}/K_{n}=\Q[x_{1},\dots, x_{m}]/
(x_{1}x_{2}\cdots x_{m}=1, x_{1}^{n}=x_{2}^{n}=\cdots x_{m}^{n}),
\]
with $W=\Sigma_{m}$ acting by {\bf{permutations}}. If $\mathcal{S}$ denotes the 
symmetrization operator, that is, 
\[
\mathcal{S}(p(x_{1},\dots,x_{m}))=\sum_{\sigma\in \Sigma_{m}}p(x_{\sigma(1)},\dots,x_{\sigma(m)}),
\]
then it can be proved that the set 
 $\{\mathcal{S}(x_{1}^{i_{1}}\cdots x_{m-1}^{i_{m-1}})\}_{0\le i_{1}\le i_{2}\le 
\cdots\le i_{m-1}\le n}$ forms a basis as a $\Q$-vector space for 
$(R(T)_{\Q}/K_{n})^{\Sigma_{m}}\cong R(G)_{\Q}/K_{n}^{\Sigma_{m}}$.

\subsection{Inertia of the commuting variety in $\g$} 

For every integer $m\ge 2$ the
commuting variety in $\g$ is defined to be
\[
C_{m}(\g)=\{(X_{1},\dots,X_{m})\in \g^{m} ~|~ [X_{i},X_{j}]=0
\text{ for all }  1\le i,j \le m \}.
\]
Let $X=C_{m}(\g)^{+}$ denote the one point compactification of the 
commuting variety $C_{m}(\g)$. The group $G$ acts on $C_{m}(\g)^{+}$ via the adjoint 
representation and fixing the point at infinity. If we assume that 
$G$ is one of the groups $SU(l)$ or $Sp(l)$ with $l\ge 2$, then the 
action of $G$ on $C_{m}(\g)^{+}$ has connected 
maximal rank isotropy subgroups and for every $x\in C_{m}(\g)^{+}$ 
the isotropy subgroups $G_{x}$ have torsion free fundamental groups 
(See \cite[Section 6.4]{AG}). Therefore the action of $G$ on $\Lambda (C_{m}(\g)^{+})$ 
has connected maximal rank isotropy subgroups 
and the hypotheses of Theorem \ref{TheoremE2-term} are satisfied. 
In this example $X^{T}=\SS^{\t^{m}}$, the one-point compactification of $\t^{m}$, with $W$ acting 
diagonally. Notice that $H^{0}(\SS^{\t^{m}};\Q)\cong \Q$ with $W$ acting 
trivially,  $H^{mr}(\SS^{\t^{m}};\Q)\cong (\Q^{sgn})^{\otimes m}$, where $W$ acts diagonally 
and by the sign representation on each factor  $\Q^{sgn}$.  Also  
$H^{p}(\SS^{\t^{m}};\Q)=0$ for all $p\ne 0, mr$. By  
Theorem \ref{TheoremE2-term}, the $E_{2}$-term of the spectral sequence computing 
${}^{Q_{n}}K_{G}^{*}(\Lambda(C_{m}(\g))^{+})\otimes \Q$ is given as follows: for $q$ even 
\[
E_{2}^{r,q}=[H^{0}(X^{T};\Q)\otimes (R(T)^{sgn}_{\Q}/J_{n})]^{W}\cong 
(R(T)^{sgn}_{\Q})^{W}/J_{n}^{W}\cong R(G)_{\Q}/L_{n}.
\]
Also, 
\[
E_{2}^{(m+1)r,q}=[H^{mr}(X^{T};\Q)\otimes (R(T)^{sgn}_{\Q}/J_{n})]^{W}\cong 
[(\Q^{sgn})^{\otimes m}\otimes R(T)^{sgn}_{\Q}/J_{n}]^{W}.
\]
For all other cases $E_{2}^{p,q}=0$. To describe this $E_{2}$-term explicitly we need to 
consider a few cases. 

Suppose first that $m$ is even. In this case we have 
$(\Q^{sgn})^{\otimes m}\cong \Q$ with the trivial $W$-action. Therefore under this assumption  
\[
E_{2}^{(m+1)r,q}\cong (R(T)^{sgn}_{\Q})^{W}/J_{n}^{W}\cong R(G)_{\Q}/L_{n}.
\]
Since $m$ is even, for dimension reasons there are no nonzero differentials and so the 
spectral sequence collapses at the $E_{2}$-term. We conclude that in this case 
\begin{equation*}
E_{\infty}^{p,q}\cong \left\{
\begin{array}{cl}
R(G)_{\Q}/L_{n}
&\text{ if } q \text{ is even and } p=r, (m+1)r\\
0& \text{ otherwise. } 
\end{array}
\right.
\end{equation*}
as a module over $R(G)_{\Q}$.

Assume now that $m$ is odd. With this assumption, we have 
$(\Q^{sgn})^{\otimes m}\cong \Q^{sgn}$ with the sign $W$-action. Therefore under this assumption  
\[
E_{2}^{(m+1)r,q}\cong (\Q^{sgn}\otimes R(T)^{sgn}_{\Q}/J_{n})^{W}\cong R(G)_{\Q}/K_{n}^{W},
\]
where as above $K_{n}$ denotes the ideal in $R(T)_{\Q}$ 
generated by $\theta_{\alpha_{i}}^{nd_{i}}-1$ for $1\le i\le r$. 
In this case the only possibly non-trivial 
differential is $d_{mr-1}$ for odd values of $r$. However, 
comparing this spectral sequence with the one computing ${}^{P_{n}}K_{G}^{*}(G)\otimes \Q$ 
it follows that such differentials must be trivial.
We conclude that in this case 
\begin{equation*}
E_{\infty}^{p,q}\cong \left\{
\begin{array}{cl}
R(G)_{\Q}/L_{n} &\text{ if } q \text{ is even and } p=r,\\
R(G)_{\Q}/K_{n}^{W}
&\text{ if } q \text{ is even and } p=(m+1)r,\\
0& \text{ otherwise, } 
\end{array}
\right.
\end{equation*}
as a module over $R(G)_{\Q}$.

In any of the above cases, by comparing the spectral sequence computing 
${}^{P_{n}}K_{G}^{p}(G)\otimes \Q$  
with that computing ${}^{Q_{n}}K_{G}^{p}(\Lambda (C_{m}(\g)^{+}))\otimes \Q$,   we see 
that there are no extension 
problems thus yielding the following corollary.

\begin{corollary}
Suppose that  $G$ is one of the groups  $SU(l)$ or $Sp(l)$ with $l\ge 2$. 
Let  $n\ne 0$ be an integer and let $r$ be the rank of $G$. Then we have
the following isomorphisms of modules over $R(G)_{\Q}$
\begin{enumerate}
\item If $m$ is even 
\begin{equation*}
{}^{Q_{n}}K_{G}^{p}(\Lambda (C_{m}(\g)^{+}))\otimes \Q\cong \left\{
\begin{array}{cl}
(R(G)_{\Q}/L_{n})^{2} 
&\text{ for  }p\equiv r\ (\text{mod } 2),\\
0
&\text{ for  } p+1\equiv r\ (\text{mod } 2).
\end{array}
\right.
\end{equation*}

\item If $m$ and $r$ are odd
\begin{equation*}
{}^{Q_{n}}K_{G}^{p}(\Lambda (C_{m}(\g)^{+}))\otimes \Q\cong \left\{
\begin{array}{cl}
R(G)_{\Q}/K_{n}^{W}
&\text{ for  } p \text{ even,} \\
R(G)_{\Q}/L_{n}
&\text{ for  } p \text{ odd.}
\end{array}
\right.
\end{equation*}

\item If $m$ is odd and $r$ is even 
\begin{equation*}
{}^{Q_{n}}K_{G}^{p}(\Lambda (C_{m}(\g)^{+}))\otimes \Q\cong \left\{
\begin{array}{cl}
R(G)_{\Q}/K_{n}^{W}\oplus R(G)_{\Q}/L_{n}
&\text{ for  } p \text{ even,} \\
0
&\text{ for  } p \text{ odd.}
\end{array}
\right.
\end{equation*}
\end{enumerate}
\end{corollary}

\subsection{Commuting pairs in $SU(2)$} 
Suppose that $X=G$ acts on itself by conjugation. 
The inertia space for this action corresponds to $\Lambda X=\Hom(\Z^{2},G)$, the space  of 
commuting pairs in $G$, equipped with the conjugation action of $G$. 
Assume that $G=SU(2)$. For this group, the space $\Hom(\Z^{2},G)$ is path-connected and  
the conjugation action has connected maximal rank isotropy by \cite[Example 2.4]{AG}.
All the hypotheses of Theorem \ref{TheoremE2-term} are met for this example 
so that we can use this theorem to compute 
${}^{Q_{n}}K^{*}_{SU(2)}(\Hom(\Z^{2},SU(2)))\otimes \Q$.
Let $T\cong \SS^{1}$ be the maximal torus  
consisting of all diagonal matrices in $SU(2)$. Then 
$X^{T}=T\cong \SS^{1}$ with $W=\Z/2=\left<\tau | \tau^{2}=1\right>$  acting by permutation 
on the diagonal entries of $T$. For this example $H^{0}(X^{T};\Q)\cong \Q$ with trivial 
$W$-action, $H^{1}(X^{T};\Q)\cong \Q^{sgn}$, where $\tau \cdot x=-x$ for all
$x\in \Q$ and $H^{p}(X^{T};\Q)=0$ for $p\ne 0, 1$. Using Theorem \ref{TheoremE2-term},  
the $E_{2}$-term of the spectral sequence computing 
${}^{Q_{n}}K^{*}_{SU(2)}(\Hom(\Z^{2},SU(2)))\otimes \Q$. is given as follows: for $q$ even 
\[
E_{2}^{1,q}=\left[ H^{0}(X^{T};\Q)\otimes (R(T)^{sgn}_{\Q}/J_{n}) \right]^{W}
\cong (R(T)^{sgn}_{\Q})^{W}/J_{n}^{W}\cong R(G)_{\Q}/L_{n},
\]
Also, 
\[
E_{2}^{2,q}=\left[ H^{1}(X^{T};\Q)\otimes (R(T)^{sgn}_{\Q}/J_{n}) \right]^{W}
\cong \left[ \Q^{sgn}\otimes (R(T)^{sgn}_{\Q}/J_{n}) \right]^{W}\cong 
R(G)_{\Q}/K_{n}^{W}.
\]
Moreover, $E_{2}^{p,q}=0$ for all other cases. 
We conclude that 
\begin{equation*}
E_{2}^{p,q}\cong \left\{
\begin{array}{llc}
R(SU(2))_{\Q}/L_{n}&\text{ if } &p=1 \text{ and } q \text{ is even},\\
R(SU(2))_{\Q}/K_{n}^{W}&\text{ if } &p=2 \text{ and } q \text{ is even},\\
0& &\text{ otherwise.}
\end{array}
\right.
\end{equation*}
The spectral sequence collapses trivially at the $E_{2}$-term without extension 
problems so that we obtain the following corollary.
\begin{corollary}
For $G=SU(2)$ and an integer $n\ne 0$ we have 
\begin{equation*}
{}^{Q_{n}}K^{p}_{SU(2)}(\Hom(\Z^{2},SU(2)))\otimes \Q\cong \left\{
\begin{array}{lcc}
R(SU(2))_{\Q}/K_{n}^{W}
&\text{ if } &p \text{ is even},\\
R(SU(2))_{\Q}/L_{n} & \text{ if } &p \text{ is odd}.
\end{array}
\right.
\end{equation*}
\end{corollary}
In this case we can find explicit bases for $R(SU(2))_{\Q}/L_{n}$ 
and $R(SU(2))_{\Q}/K_{n}^{W}$ as $\Q$-vector spaces. Suppose for simplicity that 
$n> 0$. Recall that 
\[
R(T)_{\Q}=\Q[x_{1},x_{2}]/(x_{1}x_{2}=1)
\] 
with $\tau$ permuting 
$x_{1}$ and $x_{2}$. Also, $R(SU(2))_{\Q}=R(T)_{\Q}^{W}=\Q[\sigma]$ 
with  $\sigma=x_{1}+\frac{1}{x_{1}}$. In this example 
$J_{n}^{W}$ is the $R(G)_{\Q}$-submodule in $R(T)_{\Q}^{sgn}$ 
generated by $x_{1}^{n}-x_{1}^{-n}$. Note that 
$\delta=x_{1}-x_{1}^{-1}$ so that $L_{n}$
is the ideal in $R(G)_{\Q}$ generated by the element 
\[
\sigma_{n-1}:=\frac{x_{1}^{n}-x_{1}^{-n}}{x_{1}-x_{1}^{-1}}
=x_{1}^{n-1}+x_{1}^{n-3}+\cdots+x_{1}^{3-n}+x_{1}^{1-n}.
\]
Moreover, the classes corresponding to the elements 
\[
1, x_{1}+\frac{1}{x_{1}},\dots, x_{1}^{n-2}+\frac{1}{x_{1}^{n-2}}
\]
form a free basis for $R(SU(2))_{\Q}/L_{n}$. Thus  
$R(SU(2))_{\Q}/L_{n}\cong \Q^{n-1}$ as a $\Q$-vector space. On the other hand, $K_{n}$ is the 
ideal in $R(T)_{\Q}$ generated by $x_{1}^{2n}-1$. The classes corresponding to the 
elements 
\[
1, x_{1}+\frac{1}{x_{1}},...,x_{1}^{n-1}+\frac{1}{x_{1}^{n-1}}, x_{1}^{n}
\]
form a basis for $R(SU(2))_{\Q}/K_{n}^{W}$  as a 
$\Q$-vector space. Thus  
$R(SU(2))_{\Q}/K_{n}^{W}\cong \Q^{n+1}$ as a $\Q$-vector space.  
We conclude that  
\begin{equation*}
{}^{Q_{n}}K^{p}_{SU(2)}(\Hom(\Z^{2},SU(2)))\otimes \Q\cong \left\{
\begin{array}{lcc}
\Q^{n+1}
&\text{ if } &p \text{ is even},\\\\
\Q^{n-1}&\text{ if } & p \text{ is odd},
\end{array}
\right.
\end{equation*}
as a $\Q$-vector space

\section{Appendix:The Verlinde Algebra}

In this section we recall the definition of the 
Verlinde algebra and relate it to the computations obtained 
in the previous section. Suppose that $G$ is a compact, simple and 
simply connected Lie group of rank $r$ 
with Lie algebra $\g$. Fix a maximal torus $T$ in $G$ and let $W=N_{G}(T)/T$ be 
the corresponding Weyl group. Throughout this section we will use the same notation 
as before. 

Let $LG$ denote the loop group of $G$; that is, the infinite dimensional 
group of smooth maps $\SS^{1}\to G$. Let $k\ge 0$ be 
an integer that corresponds to a level and define $V_{k}(G)$  
to be the group completion of the monoid of positive energy representations of $LG$ at level $k$. 
The group $V_{k}(G)$ equipped with the fusion product becomes a ring that is known as the 
Verlinde ring. In this section we are interested in the rational version of this ring, namely 
$V_{k}(G)_{\Q}:=V_{k}(G)\otimes \Q$ that we will refer to as the rational Verlinde algebra. 
The algebra $V_{k}(G)_{\Q}$ can alternatively 
be defined as the quotient 
\[
V_{k}(G)_{\Q}=R(G)_{\Q}/I_{k}(G),
\]
where $I_{k}(G)$ denotes the rational Verlinde ideal at level $k$.  
(See \cite[Section 4]{FHT} \cite[Section 4]{Meinrenken}). 
The rational Verlinde ideal at level $k\ge 0$ can be defined algebraically as follows.
Let us identify $R(G)$ with $R(T)^{W}=\Z[\Pi]^{W}$, where as before $\Pi=\Hom(T,\SS^{1})$.  
For each $1\le j\le r$, let $\omega_{j}^{*}:=n_{j}^{\vee}v_{j}\in \t$. 
The element $\omega_{j}^{*}$ 
represents the fundamental weight $\omega_{j}$ in the sense that 
$\omega_{j}(X)=2\pi i\left<X,\omega_{j}^{*}\right>$  
for every  $X\in \t$  and every $1\le j\le r$. Recall that
$\left<\cdot,\cdot\right>$ denotes the basic inner product in $\t$ such that  
the smallest non-zero element in $\Lambda$ has length $\sqrt{2}$.
Also, let us denote
\[
\rho^{*}:=n_{1}^{\vee}v_{1}+\cdots+n_{r}^{\vee}v_{r}
=\omega_{1}^{*}+\cdots+\omega_{r}^{*}\in \t.
\] 
Notice that 

\begin{equation*}
\alpha_{j}(\omega_{k}^{*})= \left\{
\begin{array}{ll}
(2\pi i) n_{j}^{\vee}/n_{j}=2\pi i/d_{j} &\text{ if }  k = j,\\
0 & \text{ otherwise. }
\end{array}
\right.
\end{equation*}  
Let $M_{k}$ denote the set of sequences 
of non-negative integers $(m_{1},\dots,m_{r})$ that satisfy  
$m_{1}n_{1}^{\vee}+\cdots +m_{r}n_{r}^{\vee}\le k$. With this definition 
$\{(m_{1}+1)\omega_{1}^{*}+\cdots +(m_{r}+1)\omega_{r}^{*}\}_{(m_{1},\dots, m_{r})\in M_{k}}$ 
corresponds to the set of elements in the lattice generated by 
$\omega_{1}^{*},\dots, \omega_{r}^{*}$ that 
belong to the interior of $n\mathfrak{A}(\Delta)$. 
Here as before $n=k+h^{\vee}$ and $\mathfrak{A}(\Delta)$ denotes the Weyl alcove corresponding to 
the set of simple roots $\Delta$ and that contains $0$. Equivalently, the set $M_{k}$ 
is defined in such a way that the collection 
$\{(m_{1}+1)\omega_{1}+\cdots+(m_{k}+1)\omega_{r}\}_{(m_{1},\dots, m_{r})\in M_{k}}$ 
corresponds to the set of elements in the weight lattice that 
belong to the interior of $n\mathfrak{A}^{*}(\Delta)$. Here  
$\mathfrak{A}^{*}(\Delta)$ denotes the Weyl alcove in $\Hom_{\R}(\t,i\R)$ 
that contains all the elements $f\in \Hom_{\R}(\t,i\R)$ such that 
\[
0\le \frac{f(K_{\alpha_{0}})}{2\pi i} \le 1.
\]
With our notation the rational Verlinde ideal at 
level $k$ is defined as the vanishing ideal in 
$R(T)_{\Q}^{W}$ of the set 
\[
S_{k}:
=\left\{\exp(t) \in T ~\left|~ t=\frac{1}{n}\left(m_{1}\omega_{1}^{*}+\cdots +m_{r}
\omega_{r}^{*}+\rho^{*}\right), 
\text{ where } (m_{1},...,m_{r})\in M_{k}\right\}\right. .
\]
The goal of this section is to relate the Verlinde algebra with rational coefficients  
$V_{k}(G)_{\Q}$ with the quotient 
$(R(T)^{sgn}_{\Q})^{W}/J_{n}^{W}$ that appeared in the computations provided in the previous 
section. Recall that $J_{n}$ is the ideal in $R(T)^{sgn}_{\Q}$   
generated by $\theta_{\alpha_{i}}^{nd_{i}}-1$ for $1\le i\le r$. Also recall that
$R(T)^{sgn}$ denotes the representation ring $R(T)$ with the 
signed permutation action of $W$. That is, if $w\in W$ and $p\in R(T)^{sgn}$, then  
$w\bullet p=(-1)^{\ell(w)}w\cdot p$. Consider the anti-symmetrization operator 
$\A\colon R(T)^{sgn}_{\Q}\to (R(T)^{sgn}_{\Q})^{W}$ defined by 
\[
\A(p)=\sum_{w\in W}w\bullet p=\sum_{w\in W}(-1)^{\ell(w)}w\cdot p.
\]
Clearly $\A(p)\in (R(T)^{sgn}_{\Q})^{W}$ for every $p$ and so $\A$ is well defined.
On the other hand, consider 
$\delta:=\theta_{\rho}^{-1}\prod_{\alpha>0, \alpha\in \Phi}(\theta_{\alpha}-1),$ 
where $\rho$ denotes the half sum of the positive roots. Notice that 
$\delta$ is the denominator in the Weyl character formula. 
By \cite[Proposition 2, VI \S 3 no. 3]{Bourbaki} 
we have that $\delta\in (R(T)^{sgn}_{\Q})^{W}$
and the map 
\begin{align*}
\Psi \colon (R(T)^{sgn}_{\Q})^{W}&\to (R(T)_{\Q})^{W}\cong R(G)_{\Q} \\
p&\mapsto p/\delta
\end{align*}
is an isomorphism of $R(G)_{\Q}$-modules. Let 
$L_{n}=\Psi(J_{n}^{W})\subset R(G)_{\Q}$. The map $\Psi$ 
defines an isomorphism of $R(G)_{\Q}$-modules 
\begin{equation*}
(R(T)^{sgn}_{\Q})^{W}/J_{n}^{W}\cong R(G)_{\Q}/L_{n}.
\end{equation*}

\begin{proposition}\label{Verlinde}
Suppose that $G$ is a compact, simple and simply connected Lie group. 
If $k\ge 0$ and $n=k+h^{\vee}$, then the ideal $L_{n}$ equals the rational Verlinde ideal 
$I_{k}(G)$. In particular there is an isomorphism $R(G)_{\Q}$-modules
\[
V_{k}(G)_{\Q}=R(G)_{\Q}/L_{n}\cong (R(T)^{sgn}_{\Q})^{W}/J_{n}^{W}.
\]
\end{proposition}
\begin{proof}[\bf Proof: ] Let $(m_{1},...,m_{r})\in M_{k}$ and 
$t=\frac{1}{n}\left((m_{1}+1)\omega_{1}^{*}+\cdots +
(m_{r}+1)\omega_{r}^{*}\right)$.
Then  the element $\exp(t)\in T$  is such that 
$\delta(\exp(t))\ne 0$. Therefore, via the isomorphism $\Psi$, the rational Verlinde ideal can be 
identified with the vanishing set of $S_{k}$ in 
$(R(T)^{sgn}_{\Q})^{W}$. We denote this vanishing set by $\mathcal{W}_{k}(G)$. 
To prove the proposition we need to show that $\mathcal{W}_{k}(G)=J_{n}^{W}$. For this we follow 
the next steps.

\medskip

\noindent{\bf{Step 1}:} Let us prove that $J_{n}^{W}\subset \mathcal{W}_{k}(G)$. Since 
$J_{n}$ is generated by $\theta_{\alpha_{j}}^{nd_{j}}-1$ it suffices to prove that 
each  $\theta_{\alpha_{j}}^{nd_{j}}-1$ vanishes on $S_{k}$. Indeed, suppose that 
$(m_{1},...,m_{r})\in M_{k}$ and let 
$t=\frac{1}{n}\left((m_{1}+1)\omega_{1}^{*}+\cdots +(m_{r}+1)\omega_{r}^{*}\right)$. Therefore 
\begin{align*}
\theta_{\alpha_{j}}^{nd_{j}}(\exp(t))&=e^{nd_{j}\alpha_{j}
\left(\frac{1}{n}\left((m_{1}+1)\omega_{1}^{*}+\cdots +(m_{r}+1)\omega_{r}^{*}\right)\right)}
=e^{d_{j}\alpha_{j} \left((m_{1}+1)\omega_{1}^{*}+\cdots +(m_{r}+1)\omega_{r}^{*}\right)}\\
&=e^{d_{j}\alpha_{j} \left((m_{j}+1)\omega_{j}^{*}\right)}=e^{2\pi i(m_{j}+1)}=1.
\end{align*}
This proves that $\theta_{\alpha_{j}}^{nd_{j}}-1$ vanishes on $S_{k}$  for all $1\le j\le r$ 
completing the first step.

\medskip

\noindent{\bf{Step 2}:} If $f$ is an element of the weight lattice let us denote by 
$\overline{\A}(\theta_{f})$ the class in $(R(T)^{sgn}_{\Q})^{W}/J_{n}^{W}$ corresponding to 
$\A(\theta_{f})$. Let $\mathcal{B}:=\{\overline{\A}(\theta_{(m_{1}+1)\omega_{1}+\cdots 
+(m_{r}+1)\omega_{r}})\}_{(m_{1},...,m_{r})\in M_{k}}$. As the second step we are going to 
show that the set $\mathcal{B}$ generates $(R(T)^{sgn}_{\Q})^{W}/J_{n}^{W}$  as a vector space over 
$\Q$. If this is true, then it follows that $\dim_{\Q}((R(T)^{sgn}_{\Q})^{W}/J_{n}^{W})\le |M_{k}|$.  
On the other hand, it is known that $V_{k}(G)$ is a free $\Z$-module of rank $|M_{k}|$, 
hence $\dim_{\Q}(V_{k}(G)_{\Q})=|M_{k}|$. 
We conclude that $\dim_{\Q}((R(T)^{sgn}_{\Q})^{W}/J_{n}^{W})\le \dim_{\Q}(V_{k}(G)_{\Q})$.
This together with the fact that $J_{n}^{W}\subset \mathcal{W}_{k}(G)$ proves that 
$J_{n}^{W}=\mathcal{W}_{k}(G)$ completing the proof of the proposition.

To prove step 2 notice that $R(T)_{\Q}=\Q[\Pi]$ and therefore the elements of the form 
$\A(\theta_{f})$ generate $(R(T)^{sgn}_{\Q})^{W}$ as a $\Q$-vector space, where $f$ runs 
through the weight lattice. Therefore it suffices to show that the classes 
$\overline{\A}(\theta_{f})$, for $f$ in the weight lattice, can be generated 
with the collection $\mathcal{B}$. To prove this notice that if $w\in W$ then 
\begin{equation}\label{eqn1}
\A(\theta_{f})=\pm \A(\theta_{w\cdot f}) \text{ for every element } f 
\text{ in the weight lattice.}
\end{equation} 
Next we show that $\theta_{\alpha_{0}}^{n}-1$ belongs to $J_{n}$. 
Equivalently, we will show that $\theta_{\alpha_{0}}^{n} \equiv 1 \ (\text{mod }  J_{n})$.
To see this, recall that $\alpha_{0}=n_{1}\alpha_{1}+\cdots+n_{r}\alpha_{r}
=d_{1}n_{1}^{\vee}\alpha_{1}+\cdots+d_{r}n_{r}^{\vee}\alpha_{r}$. Therefore 
$n\alpha_{0}=nd_{1}n_{1}^{\vee}\alpha_{1}+\cdots+nd_{r}n_{r}^{\vee}\alpha_{r}$. 
Using this and the fact that 
$\theta_{\alpha_{i}}^{nd_{i}} \equiv 1 \ (\text{mod }  J_{n})$ 
we get 
\begin{equation}\label{ideal1}
\theta_{\alpha_{0}}^{n}\equiv \prod_{i=1}^{r}(\theta_{\alpha_{i}}^{nd_{i}})^{n_{i}^{\vee}}
\equiv 1 \ (\text{mod }  J_{n}).
\end{equation} 
On the other hand, given an integer $m$ let 
$s_{\alpha_{0},m}\colon \Hom_{\R}(\t,i\R)\to \Hom_{\R}(\t,i\R)$ 
be the map defined by 
\[
s_{\alpha_{0},m}(f)=f-\left(\frac{f(K_{\alpha_{0}})}{2\pi i}-m\right)\alpha_{0}.
\] 
Geometrically the map $s_{\alpha_{0},m}$ corresponds to a reflection 
in $\Hom_{\R}(\t,i\R)$ with respect to the hyperplane defined by the equation 
$f(K_{\alpha_{0}})=2\pi im$. The reflection $s_{\alpha_{0},m}$  
preserves the weight lattice. Note that $s_{\alpha_{0},0}$ agrees with the reflection 
with respect to the hyperplane defined by $\alpha_{0}=0$ so that 
$s_{\alpha_{0},0}=s_{\alpha_{0}}$ belongs to the Weyl group. Also, 
$s_{\alpha_{0},mn}=s_{\alpha_{0}}+mn\alpha_{0}$ and thus
\[
\theta_{s_{\alpha_{0},mn}(f)}=\theta_{s_{\alpha_{0}}(f)}
\left(\theta_{\alpha_{0}}^{mn}-1\right)+\theta_{s_{\alpha_{0}}(f)}.
\]
Using (\ref{ideal1}) we see that 
$\theta_{s_{\alpha_{0}}(f)} \theta_{\alpha_{0}}^{mn}\equiv 
\theta_{s_{\alpha_{0}}(f)} \ (\text{mod }  J_{n})$.  This 
means that  
$\theta_{s_{\alpha_{0}}(f)} \left(\theta_{\alpha_{0}}^{mn}-1\right)\in J_{n}$ and thus
\[
\A(\theta_{s_{\alpha_{0},nm}(f)})-\A(\theta_{s_{\alpha_{0}}(f)})
=\A\left(\theta_{s_{\alpha_{0}}(f)}
\left(\theta_{\alpha_{0}}^{mn}-1\right)\right)\in J_{n}^{W}.
\]
On the other hand, we have $\A(\theta_{s_{\alpha_{0}}(f)})=-\A(\theta_{f})$.  
Therefore 
\begin{equation}\label{eqn2}
\overline{\A}(\theta_{s_{\alpha_{0},mn}(f)})=-\overline{\A}(\theta_{f})
\text{ for every element } f \text{ in the weight lattice.}
\end{equation}
If $f$ is any element in the weight lattice, we can find a unique $f'\in n\mathfrak{A}^{*}(\Delta)$ 
such that $f$ is obtained from $f'$ by applying reflections of the form $s_{\alpha_{0},mn}$ 
and acting by elements of the Weyl group. Using  (\ref{eqn1}) and (\ref{eqn2}) we conclude that 
$\overline{\A}(\theta_{f})=\pm\overline{\A}(\theta_{f'})$. This proves that if $f$ is any element 
in the weight lattice we can find a weight $f'\in n\mathfrak{A}^{*}(\Delta)$ such that 
$\overline{\A}(\theta_{f})=\pm \overline{\A}(\theta_{f'})$. To finish the proof we are going to 
show that if $f$ belongs to any of the walls of  $n\mathfrak{A}^{*}(\Delta)$ then 
$\overline{\A}(\theta_{f})=0$. If $f$ belongs to any of the walls of  
$n\mathfrak{A}^{*}(\Delta)$  that contains $0$, then 
$f$ is fixed by a reflection of the form $s_{\alpha}$ for some simple root $\alpha$. 
Let $M$ be a set of representatives of the left cosets of $W$ with respect to the 
subgroup $\{1,s_{\alpha}\}$. Therefore 
\begin{align*}
\A(\theta_{f})&=\sum_{w\in W}(-1)^{\ell(w)}\theta_{w\cdot f}
=\sum_{w\in M}(-1)^{\ell(w)}\theta_{w\cdot f}+\sum_{w\in M}(-1)^{\ell(ws_{\alpha})}
\theta_{w\cdot (s_{\alpha}\cdot f)}\\
&=\sum_{w\in M}(-1)^{\ell(w)}\theta_{w\cdot f}-\sum_{w\in M}(-1)^{\ell(w)}\theta_{w\cdot f}=0. 
\end{align*}
We conclude that if $f$ is fixed by a reflection of the form $s_{\alpha}$ then $\A(\theta_{f})=0$.
The only possibility left is that 
$f$ belongs to the wall in $n\mathfrak{A}^{*}(\Delta)$ that does not contain $0$ so 
that $f(K_{\alpha_{0}})=2\pi in$. By Lemma \ref{finallemma} below we conclude that 
in this case $\A(\theta_{f})\in J_{n}^{W}$ so that $\overline{\A}(\theta_{f})=0$. 
This finishes the proof. \qedhere
\end{proof}

\medskip

\begin{lemma}\label{finallemma} Let $G$ be as in the previous proposition and $n=k+h^{\vee}$ 
with $k\ge 0$. If $f\in \Hom_{\R}(\t,i\R)$ is such that  $f(K_{\alpha_{0}})=2\pi in$ then 
$\A(\theta_{f})\in J_{n}^{W}$.
\end{lemma}
\begin{proof}[\bf Proof: ] 
To prove this lemma we need to consider the following cases:

\medskip

\noindent{\bf{Case 1}:} $G=SU(2)$. In this case $G$ is of type $A_{1}$ and $h^{\vee}=2$.
Fix some level $k\ge 0$ and let $n=k+h^{\vee}=k+2$. Consider the maximal torus
\[
T=\left.\left\{\left[ 
\begin{array}{cc}
x_{1} & 0   \\ 
0 & x_{2}  
\end{array}
\right] \right| x_{1}, x_{2}\in \SS^{1},\ x_{1}x_{2}=1\right\}.
\]
We have $W=\Z/2=\{1,\tau\}$ with $\tau^{2}=1$. The element $\tau$ acts by permuting the 
diagonal entries of elements in $T$. The Lie algebra $\t$ is given by 
\[
\t=\left.\left\{\left[ 
\begin{array}{cc}
X_{1} &0  \\ 
0 & X_{2} 
\end{array}
\right] \right| X_{1},X_{2}\in i\R, \ X_{1}+X_{2}=0\right\}.
\]
Let $\Delta=\{\alpha_{1}=X_{1}-X_{2}=2X_{1}\}$ 
be our choice of simple roots. The fundamental weight in this example  is $\omega_{1}=X_{1}$ 
and the highest weight is $\alpha_{0}=\alpha_{1}$.
The representation ring with rational coefficients is 
$R(T)_{\Q}=\Q[x_{1},x_{2}]/(x_{1}x_{2}=1)$ and the action of
$W$ on $R(T)_{\Q}$ permutes $x_{1}$ and $x_{2}$. If $p(x_{1},x_{2})\in R(T)_{\Q}^{sgn}$, 
then $\tau\bullet p(x_{1},x_{2})=-p(x_{2},x_{1})$. In this case $J_{n}$ is the ideal in 
$R(T)_{\Q}^{sgn}$ generated by $\theta^{n}_{\alpha_{1}}-1=\frac{x_{1}^{n}}{x_{2}^{n}}-1$ 
so that $x_{2}^{n}(\theta^{n}_{\alpha_{1}}-1)=x_{1}^{n}-x_{2}^{n}\in J_{n}$. Moreover, 
$x_{1}^{n}-x_{2}^{n}\in J_{n}^{W}$ as  $\tau\bullet(x_{1}^{n}-x_{2}^{n})=x_{1}^{n}-x_{2}^{n}$. 
The only weight that belongs to the wall $n\mathfrak{A}^{*}(\Delta)$ defined by the equation 
$f(K_{\alpha_{0}})=2\pi in$ is $f=n\omega_{1}$. Note that  
$\A(\theta_{n\omega_{1}})=x_{1}^{n}-x_{2}^{n}\in J_{n}^{W}$, as we wanted to prove.

\medskip

\noindent{\bf{Case 2}:}  $G=Sp(r)$ with $r \geq 2$. In this case $G$ is a 
simple Lie group of type $C_{r}$ and $h^{\vee}=r+1$. Fix some level $k\ge 0$
and let $n=k+h^{\vee}=k+r+1$. Consider the maximal torus in $G$ consisting 
of all diagonal matrices in $Sp(n)$. Here we see $Sp(n)$ as a subgroup of $U(2n)$ in the usual 
way. Let $\Delta=\{\alpha_{1}=X_{1}-X_{2},\dots, \alpha_{r-1}=X_{r-1}-X_{r}, \alpha_{r}=2X_{r}\}$ 
be our choice of simple roots. With this choice of simple roots the highest root is 
$\alpha_{0}=2X_{1}=2\alpha_{1}+\cdots +2\alpha_{r-1}+\alpha_{r}$ so that 
$n_{1}=2,\dots, n_{r-1}=2, n_{r}=1$. Also, $n_{1}^{\vee}=n_{2}^{\vee}=\dots=n_{r}^{\vee}=1$ and 
$d_{1}=2,\dots, d_{r-1}=2, d_{r}=1$. The fundamental weights are 
\[
\omega_{1}=X_{1},\  \omega_{2}=X_{1}+X_{2},\ \dots, \omega_{r}=X_{1}+\cdots +X_{r}. 
\]
In this case $J_{n}$ 
is the ideal in $R(T)_{\Q}^{sgn}=\Q[x_{1}^{\pm 1},\dots,x_{n}^{\pm 1}]$ generated by 
\[
\theta_{\alpha_{1}}^{nd_{1}}-1=\frac{x_{1}^{2n}}
{x_{2}^{2n}}-1,\dots, \theta_{\alpha_{r-1}}^{nd_{r-1}}-1=\frac{x_{r-1}^{2n}}
{x_{r}^{2n}}-1 \text{ and } \theta_{\alpha_{r}}^{nd_{r}}-1=x_{r}^{2n}-1.
\] 
Therefore $J_{n}$ 
is also generated by $x_{1}^{2n}-1, x_{2}^{2n}-1,\dots, x_{r}^{2n}-1$. Suppose that 
$f$ is an element in the wall in $n\mathfrak{A}^{*}(\Delta)$ that does not contain $0$. 
Therefore $f$ can be written in the form $f=s_{1}\omega_{1}+\cdots+s_{r}\omega_{r}$, where 
$s_{1},\dots, s_{r}\ge 0$ are integers such that $s_{1}+\cdots +s_{r}=n=k+h^{\vee}$. 
With this notation we have 
\[
\theta_{f}=\theta_{\omega_{1}}^{s_{1}}\cdots\theta_{\omega_{r}}^{s_{r}}
=x_{1}^{a_{1}}x_{2}^{a_{2}}\cdots x_{r}^{a_{r}},
\] 
where $a_{1}=n=s_{1}+s_{2}+\cdots +s_{r}\ge a_{2}=s_{2}+\cdots +s_{r}\ge\cdots\ge a_{r}=s_{r}\ge 0$.
Note that 
\[
x_{1}^{a_{1}}x_{2}^{a_{2}}\cdots x_{r}^{a_{r}}-x_{1}^{-a_{1}}x_{2}^{a_{2}}\cdots x_{r}^{a_{r}}
=x_{1}^{-a_{1}}x_{2}^{a_{2}}\cdots x_{r}^{a_{r}}(x_{1}^{2n}-1)\in J_{n}.
\]
Thus $\A(x_{1}^{a_{1}}x_{2}^{a_{2}}\cdots x_{r}^{a_{r}})-\A(x_{1}^{-a_{1}}x_{2}^{a_{2}}
\cdots x_{r}^{a_{r}})\in J_{n}^{W}$. 
Also, $x_{1}^{-a_{1}}x_{2}^{a_{2}}\cdots x_{r}^{a_{r}}
=s_{\alpha}\cdot x_{1}^{a_{1}}x_{2}^{a_{2}}\cdots x_{r}^{a_{r}}$ with 
$\alpha=2X_{1}$. This implies that 
\[
\A(x_{1}^{a_{1}}x_{2}^{a_{2}}\cdots x_{r}^{a_{r}})=
-\A(x_{1}^{-a_{1}}x_{2}^{a_{2}}\cdots x_{r}^{a_{r}})
\] 
 We conclude that $2\A(x_{1}^{a_{1}}x_{2}^{a_{2}}\cdots x_{r}^{a_{r}})\in J_{n}^{W}$ and thus 
$\A(x_{1}^{a_{1}}x_{2}^{a_{2}}\cdots x_{r}^{a_{r}})\in J_{n}^{W}$. 
This finishes the proof in this case. 
\medskip

\noindent{\bf{Case 3}:}  Finally suppose that the group $G$ is not of type $A_{1}$ 
or of type $C_{r}$. Under this assumption we can find some $1\le j\le r$ 
(that we fix from now on) such that 
\[
2\frac{B(h_{\alpha_{j}},h_{\alpha_{0}})}{B(h_{\alpha_{j}},h_{\alpha_{j}})}=-1.
\] 
The existence of such $\alpha_{j}$ can be verified by inspecting the extended Dynkin diagram 
associated to the root system of $G$. For the groups considered in this case we can find 
some vertex corresponding to a simple root $\alpha_{j}$ that is connected by a single edge 
to the vertex corresponding to $\alpha_{0}$. This implies that the angle between 
$h_{\alpha_{0}}$ and $h_{\alpha_{j}}$ is $2\pi/3$ and therefore  
$2\frac{B(h_{\alpha_{j}},h_{\alpha_{0}})}{B(h_{\alpha_{j}},h_{\alpha_{j}})}=-1.$ Suppose that 
$f$ is an element in the weight lattice that belongs to the hyperplane defined by the equation 
$f(K_{\alpha_{0}})=2\pi in$ and let $g=f+nd_{j}\alpha_{j}$, which belongs to the weight lattice. 
The element $g$ is defined in such a way that 
\begin{align*}
g(K_{\alpha_{0}})&=f(K_{\alpha_{0}})+nd_{j}\alpha_{j}(K_{\alpha_{0}})
=2\pi in+2 \pi i nd_{j}\left( 2\frac{B(h_{\alpha_{j}},h_{\alpha_{0}})}
{B(h_{\alpha_{0}},h_{\alpha_{0}})}\right)\\
&=2\pi in+2 \pi i n\left(\frac{B(h_{\alpha_{0}},h_{\alpha_{0}})}
{B(h_{\alpha_{j}},h_{\alpha_{j}})}\right)
\left( 2\frac{B(h_{\alpha_{j}},h_{\alpha_{0}})}{B(h_{\alpha_{0}},h_{\alpha_{0}})}\right)\\
&=2\pi in+2 \pi i n\left( 2\frac{B(h_{\alpha_{j}},h_{\alpha_{0}})}
{B(h_{\alpha_{j}},h_{\alpha_{j}})}\right)
=2\pi in-2\pi in=0.
\end{align*}
Let $s_{\alpha_{0}}$ be the reflection in $W$ that corresponds to the root $\alpha_{0}$. Since 
$g(K_{\alpha_{0}})=0$ we have that $s_{\alpha_{0}}\cdot g=g$. By the same argument provided above 
we conclude that Hence $\A(\theta_{g})=0$. 
On the other hand, notice that $(\theta^{nd_{j}}_{\alpha_{j}}-1)$ belongs to
$J_{n}$ and so $\A(\theta_{f}(\theta^{nd_{j}}_{\alpha_{j}}-1))\in J_{n}^{W}$. Since 
$\theta_{g}=\theta_{f}(\theta^{nd_{j}}_{\alpha_{j}}-1)+\theta_{f}$,
we conclude that 
$$0=\A(\theta_{g})=\A(\theta_{f}(\theta^{nd_{j}}_{\alpha_{j}}-1))+\A(\theta_{f})$$
and therefore
$$\A(\theta_{f})=-\A(\theta_{f}(\theta^{nd_{j}}_{\alpha_{j}}-1))\in J_{n}^{W}.$$ \qedhere
\end{proof}

\end{document}